\documentclass[a4paper,10pt,reqno]{amsart}
\usepackage{url,latexsym,amsthm,psfrag,amsfonts,amsmath,amssymb,epsfig,hhline%,showkeys%}
}
\makeatletter
\@addtoreset{figure}{section}
\makeatother

\makeatletter
\@addtoreset{table}{section}
\makeatother

\begin{document}

\thispagestyle{empty}

\newtheorem{theorem}{Theorem}[section]
\newtheorem*{thm}{Theorem}
\newtheorem{lemma}[theorem]{Lemma}
\newtheorem{cor}[theorem]{Corollary}
\renewcommand{\proofname}{Proof}
\newtheorem{property}[theorem]{Property}
\newtheorem{prop}[theorem]{Proposition}

\theoremstyle{definition}
\newtheorem{defin}[theorem]{Definition}
\newtheorem{question}[theorem]{Question}
\newtheorem{remark}[theorem]{Remark}
\newtheorem*{nota}{Notation}
\newtheorem{example}[theorem]{Example}

\def \H{{\mathcal H}}
\def \N{{\mathbb N}}
\def \R{{\mathbb R}}
\def \X{{\mathbb X}}
\def \E{{\mathbb E}}
\def \Z{{\mathbb Z}}
\def \S{{\mathbb S}}
\def \G{{\mathcal G}}
\def \D{{\mathcal D}}
\def \P{{\mathcal P}}
\def \R{{\mathcal R}}
\def \C{{\mathcal C}}
\def \Q{{\mathcal Q}}
\def \Gr{{\mathrm Gr}}
\def \Vol{{\mathrm Vol}}
\def \l{\langle }
\def \r{\rangle }
\def \[{[ }
\def \]{] }
\def \d{D\,}
\def \t{\widetilde}
\def \sign{\text{\,sign\,}}
\def \conv{\text{\,conv\,}}
\def \wt{\widetilde}
\def \wh{\widehat}
\def \a{\alpha}
\def \b{\beta}
\def \dim{\text{\,dim\,}}
\def \rank{\text{\,rank\,}}
\def \Symb{\text{\,Simb\,}}
\def \relint{\text{rel\.int\,}}

\newcommand{\arcsinh}{\mathop{\mathrm{arcsinh}}\nolimits}
\newcommand{\vn}{\mathop{\mathrm{int}}\nolimits}
\newcommand{\rel}{\mathop{\mathrm{rel\:int}}\nolimits}
\newcommand{\w}{\widetilde }
\renewcommand{\o}{\overline }

%\author{Anna Felikson}
%\address{Independent University of Moscow, B. Vlassievskii 11, 119002 Moscow, Russia}
%\curraddr{Department of Mathematics, University of Fribourg, P\'erolles, Chemin du Mus\'ee 23, CH-1700 Fribourg, Switzerland}
%\email{felikson@mccme.ru}
%\thanks{Research supported by RFBR 07-01-00390-a (A.~F. and P.~T.),...}

%\author{Pavel Tumarkin}
%\address{Independent University of Moscow, B. Vlassievskii 11, 119002 Moscow, Russia}
%\curraddr{Department of Mathematics, University of Fribourg, P\'erolles, Chemin du Mus\'ee 23, CH-1700 Fribourg, Switzerland}
%\email{pasha@mccme.ru}

\title[Coxeter groups from cluster algebras]{Coxeter groups and their quotients \\ arising from cluster algebras }
%\title{Presentations of affine Coxeter groups from cluster algebras}

\author{Anna Felikson and Pavel Tumarkin}
\address{Department of Mathematical Sciences, Durham University, Science Laboratories, South Road, Durham, DH1 3LE, UK}
\email{anna.felikson@durham.ac.uk, pavel.tumarkin@durham.ac.uk}
\thanks{Research was supported in part by grant RFBR 11-01-00289-a}
\dedicatory{To the memory of Andrei Zelevinsky}

\begin{abstract} 
In~\cite{BM}, Barot and Marsh presented an explicit construction of presentation of a finite Weyl group $W$ by any initial seed of corresponding cluster algebra, i.e. by any diagram mutation-equivalent to an orientation of a Dynkin diagram with Weyl group $W$. We obtain similar presentations for all affine Coxeter groups.
Furthermore, we generalize the construction to the settings of diagrams arising from unpunctured triangulated surfaces and orbifolds, which leads to presentations of corresponding groups as quotients of numerous distinct Coxeter groups.

% We also present a geometric interpretation of such representations in terms of piecewise simplicial complexes with proper $G$-action. For certain finite Weyl groups this leads to manifolds with affine or hyperbolic structure and proper $G$-action. 
\end{abstract}

\maketitle
\tableofcontents

\section{Introduction}

In~\cite{FZ2}, Fomin and Zelevinsky provide a classification of cluster algebras of finite type: they show that these cluster algebras are classified by Dynkin diagrams, and there is one-to-one correspondence between cluster variables on one side, and positive roots and negatives of simple roots on the other side. 
In the same paper, Fomin and Zelevinsky associate to every seed of a skew-symmetrizable cluster algebra a {\it diagram} constructed by the corresponding exchange matrix; mutations of these diagrams encode the mutations of exchange matrices. 

Starting from an arbitrary diagram of a cluster algebra of finite type, Barot and Marsh~\cite{BM} provide a presentation of the corresponding finite Weyl group. The construction works as follows: one needs to consider the underlying unoriented labeled graph of a diagram as a Coxeter diagram of a Coxeter group, and then introduce some additional relations on this group that can be read off from the diagram. These additional relations come from {\it oriented cycles} of the diagram and can be written as follows: for any chordless oriented cycle
$$i_0\stackrel{w_1}{\to}i_1\stackrel{w_2}{\to}\dots \stackrel{w_{d-1}}{\to}i_{d-1}\stackrel{w_0}{\to}i_0$$
in the diagram, where either $w_0=2$ or all $w_i=1$, we have
$$(s_{i_0}s_{i_1}\dots s_{i_{d-2}}s_{i_{d-1}}s_{i_{d-2}}\dots s_{i_{1}})^2=e.$$
The resulting group occurs to depend on the mutation class of the diagram only.

The presentations of finite Weyl groups as quotients of other Coxeter groups lead to interesting consequences. For example, in~\cite{FeTu} these presentations are used to construct hyperbolic manifolds having large symmetry groups and relatively small volumes. Further, the construction of Barot and Marsh implies that for every Weyl group there exists a distinguished set of generating tuples of reflections (the collections of corresponding roots are called {\it companion bases} in~\cite{P} and then in~\cite{BM}). According to results of~\cite{Fe}, the companion bases do not exhaust all the minimal generating tuples of relections of a Weyl group. The question whether there is a geometric characterization of companion bases is really intriguing.    

\medskip

The aim of the present paper is to obtain similar results for affine Weyl groups
and to generalize the construction to the case of diagrams arising from unpunctured triangulated surfaces and orbifolds.

%We also present some geometric interpretation of these presentations of Coxeter groups. 
%In particular, this leads to 

\bigskip

\noindent
Let $\widetilde \G$ be an orientation of an affine Dynkin diagram with $n+1$ nodes different from an oriented cycle, let $W$ be the corresponding affine Coxeter group, and $\G$ be any diagram mutation-equivalent to $\widetilde \G$. Denote by $W_{\G}$ the group generated by $n+1$ generators $s_i$ with the following relations:

\begin{itemize}
\item[(1)] $s_i^2=e$ for all $i=1,\dots,n$;

\item[(2)] $(s_is_j)^{m_{ij}}=e$ for all $i,j$ not joined by an edge labeled by $4$, where
$$
m_{ij}=
\begin{cases}
2 & \text{if $i$ and $j$ are not joined;} \\
3 & \text{if $i$ and $j$ are joined by an edge labeled by $1$;} \\
4 & \text{if $i$ and $j$ are joined by an edge labeled by $2$;} \\
6 & \text{if $i$ and $j$ are joined by an edge labeled by $3$.} 
\end{cases}
$$

\item[(3)] (cycle relation) for every chordless oriented cycle $\C$ given by 
$$i_0\stackrel{w_{i_0i_1}}{\to} i_1\stackrel{w_{i_1i_2}}{\to}\cdots\stackrel{w_{i_{d-2}i_{d-1}}}{\to} i_{d-1}\stackrel{w_{i_{d-1}i_0}}{\to}i_0$$
and for every $l=0,\dots,d-1$ we define a number 
$$t(l)=\left(\prod\limits_{j=l}^{l+d-2}\!\!\!\!\sqrt{w_{i_ji_{j+1}}}\ -\sqrt{w_{i_{l+d-1}i_l}}\right)^2,$$ 
where the indices are considered modulo $d$;
now for every $l$ such that $t(l)<4$, we take a relation
$$
(s_{i_l}s_{i_{l+1}}\dots s_{i_{l+d-1}}s_{i_{l+d-2}}\dots s_{i_{l+1}})^{m(l)}=e,
$$
where 
$$
m(l)=
\begin{cases}
2 & \text{if $t(l)=0$;} \\
3 & \text{if $t(l)=1$;} \\
4 & \text{if $t(l)=2$;} \\
6 & \text{if $t(l)=3$} 
\end{cases}
$$
(this form of cycle relations was introduced by Seven in~\cite{S}).

\item[(4)] (additional affine relations) for every subdiagram of $\G$ of the form shown in the first column of Table~\ref{add-t} we take the relations listed in the second column of the table. 
\end{itemize}

%$\bullet$ For any chordless oriented cycle in $\D$ without edges of weight $4$
%$$i_0\stackrel{w_1}{\to}i_1\stackrel{w_2}{\to}\dots \stackrel{w_{d-1}}{\to}i_{d-1}\stackrel{w_0}{\to}i_0,$$
%where either $w_0=2$ or all $w_i=1$, we have
%$$(s_{i_0}s_{i_1}\dots s_{i_{d-2}}s_{i_{d-1}}s_{i_{d-2}}\dots s_{i_{1}})^2=1$$
%
%$\bullet$ For any chordless oriented cycle in $\D$ 
%$$i_0\stackrel{w_1}{\to}i_1\stackrel{w_2}{\to}i_{2}\stackrel{4}{\to}i_0,$$
%we have
%$$(s_{i_0}s_{i_1}s_{i_{2}}s_{i_1})^{4-w_1}=1$$
 
The group constructed does not depend on the choice of a diagram in the mutation class of $\widetilde \G$ and
is isomorphic to the initial affine Coxeter group $W$:

\bigskip
\noindent
{\bf Theorem~\ref{thm-aff}}.
{\it
The group $W_{\G}$ is isomorphic to $W$.
}

\medskip
\noindent
In particular, $W_{\G}$ is an affine Coxeter group and is invariant under mutations of the diagram $\G$.

\bigskip
\noindent
As a next step, we want to generalize the construction to the case of mutation-finite diagrams.
Any mutation-finite diagram of order bigger than 2 is either one arising from a triangulated surface/orbifold
or one of the finitely many exceptional mutation types (see Theorem~\ref{class}).

In this paper, we consider the case of unpunctured triangulated surfaces and orbifolds
as well as all the exceptional finite mutation types.
%Diagrams arising from punctured surfaces and orbifolds will be considered in an upcoming paper.
The definition of a group $W_\G$ for a diagram $\G$ arising from unpunctured surface or orbifold 
(see Definition~\ref{def surf})
includes relations (1)--(4) above as well as two more relations corresponding to triangulated handles attached to the surface (or orbifold):

\begin{itemize}
\item[(5)] (additional relations for a handle) 
$$
(s_{1}s_{2}s_{3}s_{4}s_{3}s_{2})^3=e
\text{ \quad and \quad }
(s_{1}s_{2}s_{3}s_{4}s_{5}s_{4}s_{3}s_{2})^2=e
$$
for all subdiagrams  of type $\H_0$ and $\H$ shown in Fig.~\ref{rel-handle}.

\end{itemize}

Surprisingly, this small addition to the affine version of the definition is sufficient for the invariance of the group:

\bigskip
\noindent
{\bf Theorem~\ref{thm surf/orb}}.
{\it
If $\G$ is a diagram arising from an unpunctured surface or orbifold and 
 $W_{\G}$ is a group defined as above, then $W_{\G}$ is invariant under the mutations of $\G$.
}

\bigskip
%\noindent 
In contrast to the groups defined by diagrams of finite and affine types, in the case of diagrams arising from
surfaces or orbifolds the group $W_\G$ is usually not a Coxeter group but a quotient of a Coxeter group.

\bigskip
%\noindent
It turns out that in the case of exceptional diagrams one can use almost the same definition
of the group $W_\G$ as in the affine case: we only add one additional relation
$$(s_1s_0s_2s_0s_1s_3s_0s_4s_0s_3)^2=e$$
for the diagram $X_5$ shown in Fig.~\ref{X5}.

\bigskip
\noindent
{\bf Theorem~\ref{thm exc}}.
{\it
If $\G$ is a diagram of the exceptional finite mutation type 
(i.e. $\G$ is mutation-equivalent to one of $X_6,X_7,E_6^{(1,1)},E_7^{(1,1)},E_8^{(1,1)},G_2^{(*,+)},G_2^{(*,*)},F_4^{(*,+)}$ and 
$F_4^{(*,*)}$, see Table~\ref{exceptional})
then  the group  $W_{\G}$ is invariant under mutations of $\G$.
}

\bigskip
\noindent
As for diagrams arising from surfaces or orbifolds, the group
defined is not a Coxeter group but a quotient of a Coxeter group.

\bigskip

\noindent
The paper is organized as follows. In Section~\ref{bckgr}, we collect
preliminaries: we define mutations of diagrams, diagrams of finite,
affine and finite mutation type, we also discuss
diagrams arising from triangulated surfaces and orbifolds and their
block decompositions.   
In Section~\ref{subd}, we prove some auxiliary technical facts about
subdiagrams of mutation-finite diagrams.

In Section~\ref{aff group def}, we construct the group $W_\G$ for an affine diagram $\G$.
As it is explained in Section~\ref{redundance}, our definition contains some
redundant relations, which are excluded in the same section. Section~\ref{invariance}
is devoted to the proof of Theorem~\ref{thm-aff}. The proof mainly follows
one from~\cite{BM}, however we try to substitute computations by
geometric arguments coming from surface or orbifold presentations whenever possible.
In Section~\ref{difference}, we show that the additional affine relations are
essential in the sense that without these relations the group
$W_{\G}$ would not be invariant under mutations.

In Section~\ref{unpunctured}, we extend the construction of  the group $W_\G$ to the
case of diagrams arising from triangulated surfaces and orbifolds and
prove the invariance of the groups obtained (Theorem~\ref{thm surf/orb}).

Finally, in Section~\ref{except} we construct the group $W_{\G}$ for all
exceptional diagrams and prove invariance of this group under mutations (Theorem~\ref{thm exc}).

\bigskip

We are grateful to Robert Marsh for helpful discussions.  We also thank Aslak Buan and Robert Marsh for communicating to us a representation-theoretic proof of the skew-symmetric version of Lemma~\ref{subd of aff}.  
Most of the work was carried out during the program on cluster algebras at MSRI in the Fall of 2012. We would like to thank the organizers of the program for invitation, and the Institute for hospitality and excellent research atmosphere. We would also like to thank the referees for valuable comments and suggestions.

\section{Cluster algebras and diagrams of  finite mutation type}
\label{bckgr}
 In this section, we recall the essential notions on cluster algebras of finite, affine, and finite mutation type.
For details see e.g~\cite{FZ2} and~\cite{FeSTu3}.

\subsection{Diagrams and mutations}
\label{dm}
A coefficient-free cluster algebra is completely defined by a skew-symmetrizable integer matrix. Following~\cite{FZ2}, we encode an $n\times n$ skew-symmetrizable integer matrix $B$ by a finite simplicial $1$-complex $\G$ with oriented weighted edges (called {\it arrows}), and call this complex a {\it diagram}. The weights of a diagram are positive integers.

Vertices of $\G$ are labeled by $[1,\dots,n]$. If $b_{ij}>0$, we join vertices $i$ and $j$ by an arrow directed from $i$ to $j$ and assign to it weight $-b_{ij}b_{ji}$. All such diagrams satisfy the following property: a product of weights along any chordless cycle of $\G$ should be a perfect square (cf.~\cite[Exercise~2.1]{Kac}).

Throughout the paper we assume that all diagrams are connected (equivalently, matrix $B$ is assumed to be indecomposable). 

\begin{remark}
We say that arrows labeled by $1$ are {\it simple} and omit the label on the diagrams.
The diagram is {\it simply-laced}  if it contains no non-simple arrows.

\end{remark}

Distinct matrices may have the same diagram. At the same time, it is easy to see that only finitely many matrices may correspond to the same diagram.
All the weights of a diagram of a skew-symmetric matrix are perfect squares. Conversely, if all the weights of a diagram $\G$ are perfect squares, then there exists a skew-symmetric matrix $B$ with diagram $\G$. 

For every vertex $k$ of a diagram $\G$ one can define an involutive operation  $\mu_k$ called {\it mutation of $\G$ in direction $k$}. This operation produces a new diagram  denoted by $\mu_k(\G)$ which can be obtained from $\G$ in the following way (see~\cite{FZ2}): 
\begin{itemize}
\item
orientations of all arrows incident to a vertex $k$ are reversed; 
\item
for every pair of vertices $(i,j)$ such that $\G$ contains arrows directed from $i$ to $k$ and from $k$ to $j$ the weight of the arrow joining $i$ and $j$ changes as described in Figure~\ref{quivermut}.
\end{itemize} 

\begin{figure}[!h]
\begin{center}
\psfrag{a}{\small $a$}
\psfrag{b}{\small $b$}
\psfrag{c}{\small $c$}
\psfrag{d}{\small $d$}
\psfrag{k}{\small $k$}
\psfrag{mu}{\small $\mu_k$}
\epsfig{file=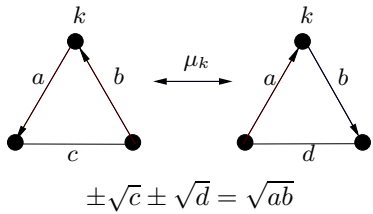,width=0.3\linewidth}\\
\medskip
$\pm\sqrt{c}\pm\sqrt{d}=\sqrt{ab}$
\caption{Mutations of diagrams. The sign before $\sqrt{c}$ (resp., $\sqrt{d}$) is positive if the three vertices form an oriented cycle, and negative otherwise. Either $c$ or $d$ may vanish. If $ab$ is equal to zero then neither value of $c$ nor orientation of the corresponding arrow does change.}
\label{quivermut}

\end{center}
\end{figure}

Given a diagram $\G$, its {\it mutation class} is the set of all diagrams obtained from
the given one by all sequences of iterated mutations. All diagrams from one mutation class are called {\it mutation-equivalent}.

%We call a diagram (resp., matrix) {\it mutation-finite} if its mutation class is finite.

\subsection{Finite type}

A diagram is of {\it finite type} if it is mutation-equivalent to an orientation of a Dynkin diagram.
So, a diagram of finite type is of one of the following mutation types: 
$A_n$, $B_n=C_n$, $D_n$, $E_6$, $E_7$, $E_8$, $F_4$ or $G_2$ (see the left column in Table~\ref{fin and aff}).

It is shown in~\cite{FZ2} that mutation classes of diagrams of finite type are in one-to-one correspondence
with cluster algebras of finite type. In particular, this implies that any subdiagram of a diagram of finite type is also of finite type.

\subsection{Affine type}
A diagram is of {\it affine type} if it is mutation-equivalent to an orientation of an affine Dynkin diagram different from an oriented cycle.
A diagram of affine type is of one of the following mutation types: 
$\widetilde A_{k,n-k}$, $0<k<n$ (see Remark~\ref{a}), $\widetilde B_n$, $\widetilde C_n$, $\widetilde D_n$, $\widetilde E_6$, $\widetilde E_7$, $\widetilde E_8$, $\widetilde  F_4$ 
or $\widetilde G_2$ (see the right column in Table~\ref{fin and aff}).

\begin{remark}
\label{a}
Let $\widetilde \G$ be an affine Dynkin diagram different from $\widetilde A_n$. Then all orientations of $\widetilde \G$ are mutation-equivalent. The orientations of  $\widetilde A_{n-1}$ split into $[n/2]$ mutation classes $\widetilde A_{k,n-k}$ (each class contains a cyclic representative with only two changes of orientations, as in Table~\ref{fin and aff}, with $k$ consecutive arrows in one direction and $n-k$ in the other, $0<k<n$).

\end{remark} 

We will heavily use the following statement. 
\begin{lemma}
\label{subd of aff}
Any subdiagram of a diagram of affine type is either of finite or of affine type.

\end{lemma} 
In skew-symmetric case Lemma~\ref{subd of aff} can be derived from the results of~\cite{BMR} and~\cite{Z}. In general case Lemma~\ref{subd of aff} immediately follows from~\cite[Theorem~1.1]{FeSThTu}.

\begin{table}[!h]
\caption{Diagrams of finite and affine type}
\label{fin and aff}
\begin{center}
\begin{tabular}{|c|c|c|c|}
\hline
\multicolumn{2}{|c|}{Finite types}&\multicolumn{2}{|c|}{Affine types}\\
\hline
\raisebox{4mm}{$A_n$, $n\ge 1$} &\raisebox{3mm}{\epsfig{file=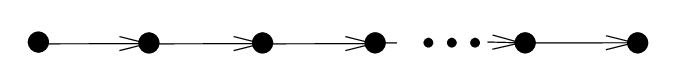,width=0.25\linewidth}} &
\raisebox{4mm}{$\widetilde A_{k,n-k}$, $n>k\ge 1$} & \raisebox{0mm}{\epsfig{file=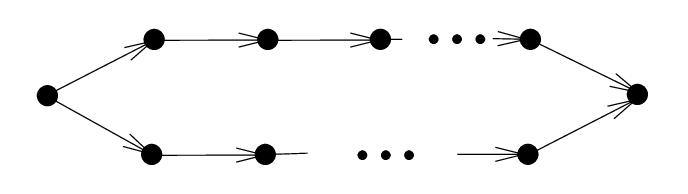,width=0.25\linewidth}}\\
\hline
\raisebox{1mm}{$B_n=C_n$, $n\ge 2$} &\psfrag{2}{\scriptsize $2$}\epsfig{file=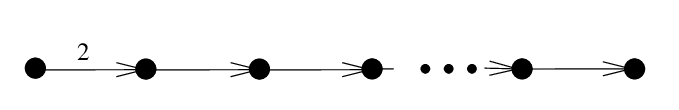,width=0.25\linewidth} &
\raisebox{5mm}{$\widetilde  B_n$, $n\ge 3$} & \psfrag{2}{\scriptsize $2$}\epsfig{file=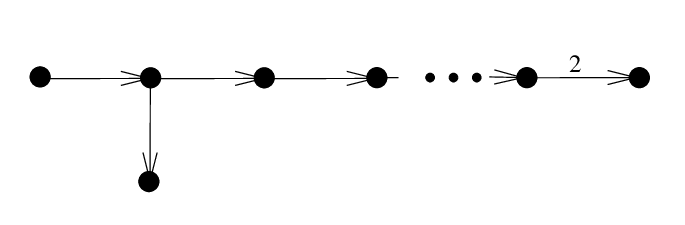,width=0.25\linewidth}\\
\hhline{|~|~|-|-|}
&&\raisebox{2mm}{$\widetilde C_n$, $n\ge 2$} &
 \psfrag{2}{\scriptsize $2$}\epsfig{file=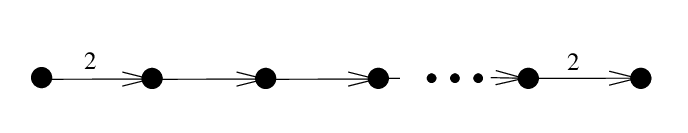,width=0.25\linewidth}\\
\hline
\raisebox{5mm}{$D_n$, $n\ge 4$}&\epsfig{file=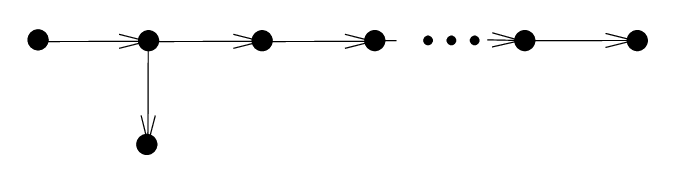,width=0.25\linewidth} &
\raisebox{5mm}{$\widetilde D_n$, $n\ge 4$}& \epsfig{file=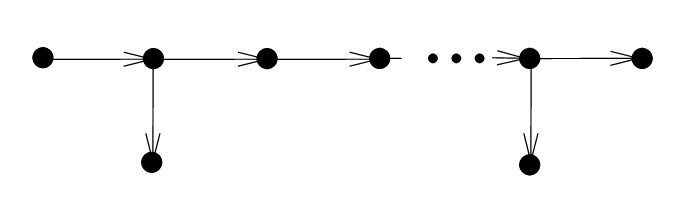,width=0.25\linewidth}\\
\hline
\raisebox{7mm}{$E_6$}& \raisebox{4mm}{\epsfig{file=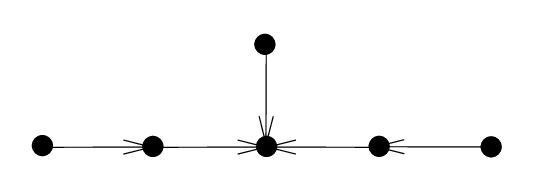,width=0.2\linewidth}}&
\raisebox{7mm}{$\widetilde E_6$}& \epsfig{file=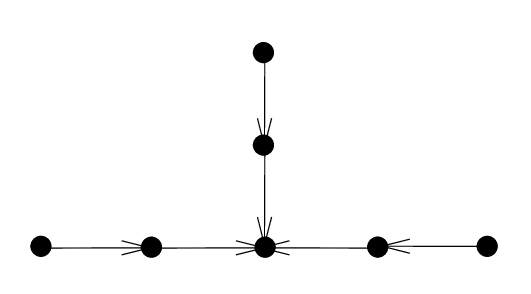,width=0.2\linewidth}\\
\hline
\raisebox{2mm}{$E_7$}& \epsfig{file=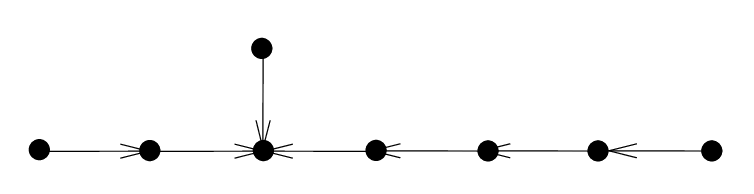,width=0.25\linewidth}&
\raisebox{2mm}{$\widetilde E_7$}&\epsfig{file=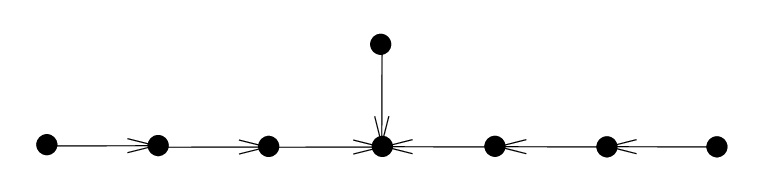,width=0.25\linewidth}\\
\hline
\raisebox{2mm}{$E_8$}& \epsfig{file=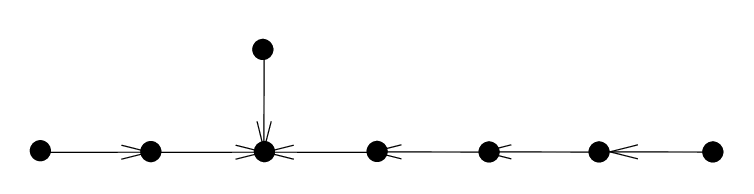,width=0.25\linewidth}&
\raisebox{2mm}{$\widetilde E_8$}&\epsfig{file=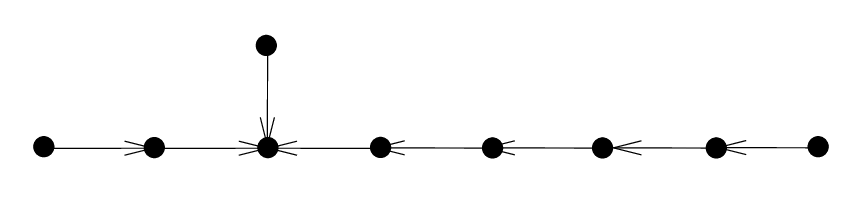,width=0.25\linewidth}\\
\hline
\raisebox{2mm}{$F_4$}&\psfrag{2}{\scriptsize $2$}\epsfig{file=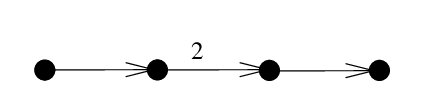,width=0.15\linewidth} &
\raisebox{2mm}{$\widetilde F_4$}&\psfrag{2}{\scriptsize $2$}\epsfig{file=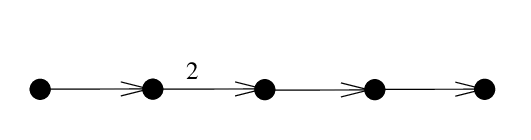,width=0.20\linewidth}\\
\hline
\raisebox{2mm}{$G_2$}& \psfrag{3}{\scriptsize $3$}\epsfig{file=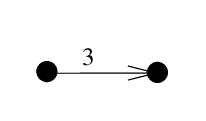,width=0.08\linewidth}&
\raisebox{2mm}{$\widetilde G_2$}&\psfrag{3}{\scriptsize $3$}\raisebox{1mm}{\epsfig{file=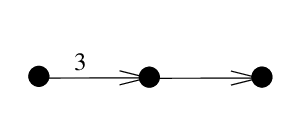,width=0.10\linewidth}}\\
%\hline
%\psfrag{3}{$E_8$}
%\epsfig{file=./pic/fin.eps,width=0.89\linewidth}\\
\hline
\end{tabular}
\end{center}
\end{table}

\subsection{Finite mutation type }

A diagram  is called {\it mutation-finite} (or {\it of finite mutation type}) if its mutation class is finite.

The following criterion for a diagram to be mutation-finite is well-known (see e.g.~\cite[Theorem 2.8]{FeSTu2}).

\begin{prop}
A diagram $\G$ of order at least $3$ is mutation-finite if and only if any diagram in the
mutation class of $\G$ contains no arrows of weight greater than $4$.

\end{prop}

Mutation-finite diagrams of order at least $3$ containing no arrows of weight $2$ and $3$ will be called {\it skew-symmetric}
(as for any of them there is the corresponding skew-symmetric matrix).

As it is shown in~\cite{FeSTu1},~\cite{FeSTu2} and~\cite{FeSTu3}, a diagram of finite mutation type
either has only two vertices, or corresponds to a triangulated surface or orbifold (see Section~\ref{triang}), or belongs to one of finitely many exceptional mutation classes.

\begin{theorem}[\cite{FeSTu1,FeSTu2,FeSTu3}]
\label{class}
Let $\G$ be a mutation-finite diagram with at least $3$ vertices. Then either $\G$ arises from a triangulated surface or orbifold, or $\G$ is mutation-equivalent to one of $18$ exceptional diagrams
$E_6,E_7,E_8, \widetilde E_6,\widetilde E_7,\widetilde E_8,E_6^{(1,1)}\!,E_7^{(1,1)}\!,E_8^{(1,1)}\!,X_6,X_7,
\widetilde G_2,G_2^{(*,+)}\!,G_2^{(*,*)}\!, F_4, \widetilde F_4, F_4^{(*,+)}\!$ or  $F_4^{(*,*)}$
shown in Fig.~\ref{exceptional}.

\end{theorem}

\begin{table}[!h]
\caption{Exceptional mutation classes}
\label{exceptional}
\begin{center}
\psfrag{a}{\small $a$}
\begin{tabular}{|c|}
\hline
\\
\underline{\large Skew-symmetric diagrams}:\\
\\
\psfrag{1}{$E_6$}
\psfrag{2}{$E_7$}
\psfrag{3}{$E_8$}
\psfrag{1_}{$\widetilde E_6$}
\psfrag{2_}{$\widetilde E_7$}
\psfrag{3_}{$\widetilde E_8$}
\psfrag{1__}{$E_6^{(1,1)}$}
\psfrag{2__}{$E_7^{(1,1)}$}
\psfrag{3__}{$E_8^{(1,1)}$}
\psfrag{4}{$4$}
\psfrag{4-}{$X_6$}
\psfrag{5-}{$X_7$}
\epsfig{file=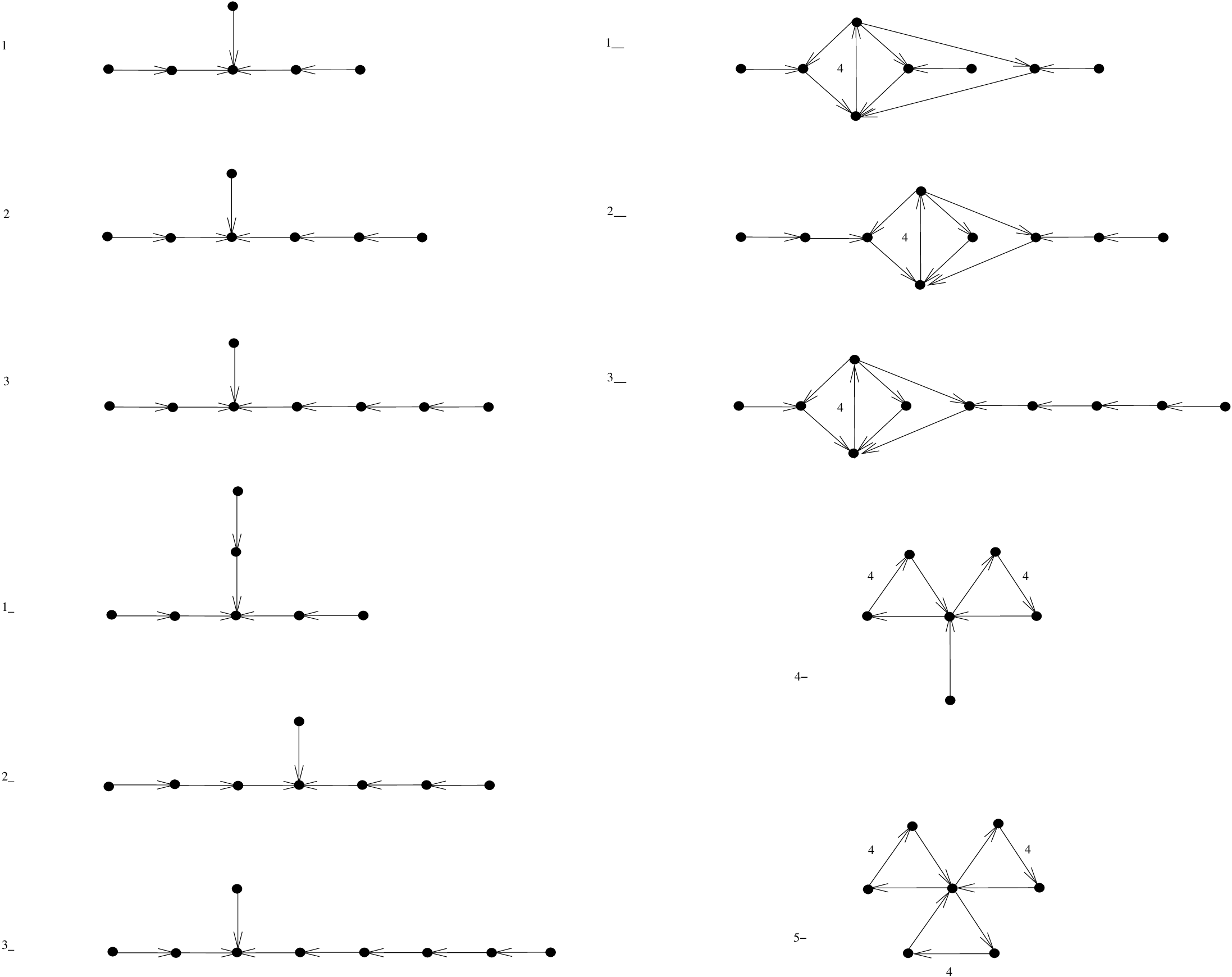,width=0.89\linewidth}\\
\\
\hline
\\
\underline{\large Non-skew-symmetric diagrams}:\\
\\
\psfrag{2}{\scriptsize $2$}
\psfrag{2-}{\scriptsize $2$}
\psfrag{3}{\scriptsize $3$}
\psfrag{4}{\scriptsize $4$}
\psfrag{G}{$\t G_2$}
\psfrag{F}{$F_4$}
\psfrag{Ft}{$\t F_4$}
\psfrag{W}{$G_2^{(*,*)}$}
\psfrag{V}{$G_2^{(*,+)}$}
%\psfrag{Y5}{$Y_5$}
\psfrag{Y6}{$F_4^{(*,+)}$}
\psfrag{Z}{$F_4^{(*,*)}$}
\epsfig{file=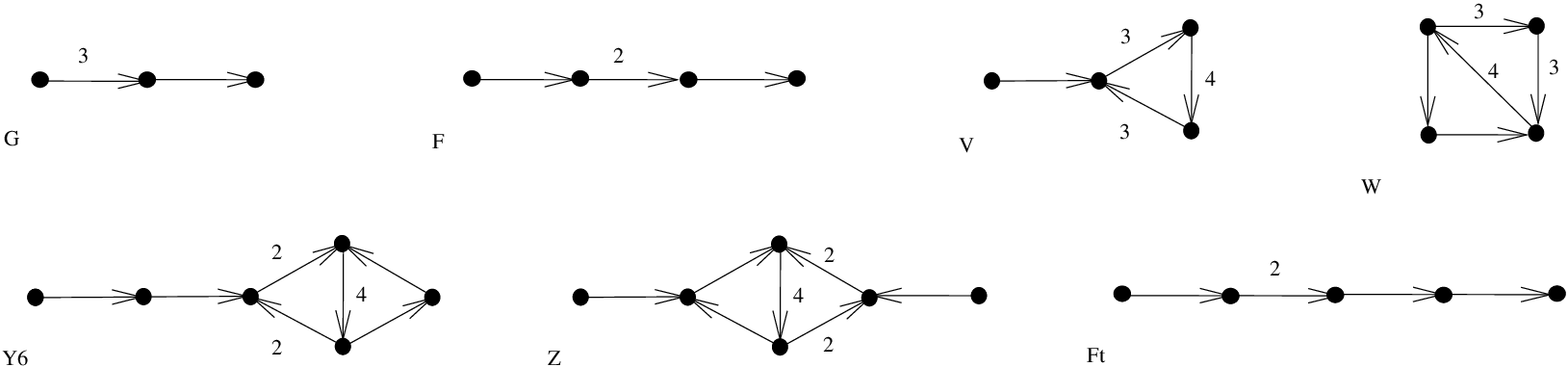,width=0.89\linewidth}\\
\hline
\end{tabular}
\end{center}
\end{table}

\subsection{Triangulated surfaces/orbifolds  and block-decomposable diagrams}
\label{triang}

The correspondence between diagrams of finite mutation type and triangulated surfaces (or orbifolds with orbifold points of order $2$) is developed in~\cite{FST} and~\cite{FeSTu3}. Here we briefly remind the basic definitions.

By a {\it surface} we mean a genus $g$ orientable surface with $r$ boundary components and a finite set of marked points, with at least one marked point at each boundary component. A non-boundary marked point is called a {\it puncture}. By an {\it orbifold} we mean a surface with a distinguished finite set of interior points called {\it orbifold points of order $2$}.

An (ideal) {\it triangulation} of a surface is a triangulation with vertices of triangles in the marked points. We allow self-folded triangles and  follow~\cite{FST} considering triangulations as {\it tagged triangulations} (however, we are neither reproducing nor using all the details in this paper).
%You can find the list of possible triangles in a triangulation of a surface in Table~\ref{blocks}, column ``surface''.

An (ideal) {\it triangulation} of an orbifold is constructed similarly to a triangulation of a surface,
but it also includes ``orbifold triangles'' (see Table~\ref{blocks}).
In these triangles a cross stays for an orbifold point. An edge of the triangulation incident to an orbifold point 
is called a {\it pending edge}, it is drawn bold and is 
thought as a round-trip from a ordinary marked point to the orbifold point and back. 

Given a triangulated surface or orbifold, one constructs a diagram in the following way:
\begin{itemize}
\item vertices of the diagram correspond to the (non-boundary) edges of a triangulation;
\item two vertices are connected by a simple arrow if they correspond to two sides of the same triangle (i.e., there is one simple arrow between given two vertices for every such triangle);
inside the triangle orientations of the arrow are arranged counter-clockwise (with respect to some orientation of the surface);
\item two simple arrows with different directions connecting the same vertices cancel out;
two simple arrows in the same direction add to an arrow of weight $4$;
\item an arrow between vertices corresponding to a pending edge and an ordinary edge of a triangle has weight $2$;
an arrow between two vertices corresponding to two pending edges has weight $4$.
\item for a self-folded triangle (with two sides identified), two vertices corresponding to the sides of this triangle are disjoint;
a vertex corresponding to the ``inner'' side of the triangle is connected to other vertices in the same way as the vertex corresponding to the outer side of the triangle.

\end{itemize}

It is easy to see that any surface (or orbifold) can be cut into {\it elementary surfaces/orbifolds}, we list them (and their diagrams) in Table~\ref{blocks}.  
We use {\it white} color for the vertices corresponding to the ``exterior'' edges of these elementary surfaces and {\it black} for the vertices corresponding to ``interior'' edges.

The diagrams in  Table~\ref{blocks} are called {\it blocks}.
We will say that blocks listed on the left are {\it skew-symmetric} ones, while the ones on the right are {\it non-skew-symmetric}. Depending on a block, we call it {\it a block of type} ${\rm{I}}$, ${\rm{II}}$ etc. (see the left column of Table~\ref{blocks}).

As elementary surfaces and orbifolds are glued to each other to form a triangulated surface or orbifold, the blocks are glued to form a {\it block-decomposition} of a bigger diagram. 
A connected diagram $\G$ is called {\it block-decomposable} (or simply, {\it decomposable})
if it can be obtained from a collection of blocks by identifying white vertices of different blocks along some partial matching (matching of vertices of the same block is not allowed), where two simple arrows with same endpoints and opposite directions cancel out, and two simple arrows with same endpoints and same directions form an arrow of weight $4$. A non-connected diagram $\G$ is called  block-decomposable if every connected component of $\G$ is either decomposable or a single vertex. If a diagram $\G$ is not block-decomposable then we call $\G$ {\it non-decomposable}.

\begin{table}[!h]
\begin{center}
\caption{Elementary surfaces/orbifolds and corresponding blocks}
\label{blocks}
\begin{tabular}{cc}
\begin{tabular}{|c|c|c|}
\hline
Type& \phantom{x} Diagram \phantom{x} & Surface\\
\hline 
\raisebox{9mm}{${\rm{I}}$}& \raisebox{8.5mm}{\epsfig{file=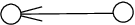,width=0.1\linewidth}}& \epsfig{file=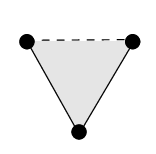,width=0.15\linewidth}\\
\hline
\raisebox{9mm}{${\rm{II}}$}& \raisebox{3mm}{\epsfig{file=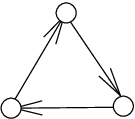,width=0.1\linewidth}} & \epsfig{file=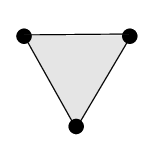,width=0.15\linewidth}\\
\hline
\raisebox{9mm}{${\rm{IIIa}}$}&  \raisebox{5mm}{\epsfig{file=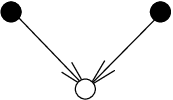,width=0.1\linewidth}}& \epsfig{file=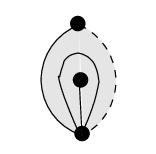,width=0.15\linewidth}\\
\hline
\raisebox{9mm}{${\rm{IIIb}}$}&  \raisebox{5mm}{\epsfig{file=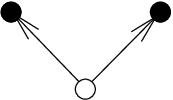,width=0.1\linewidth}}& \epsfig{file=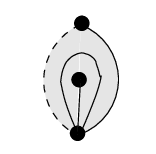,width=0.15\linewidth}\\
\hline
\raisebox{9mm}{${\rm{IV}}$}&  \raisebox{1mm}{\epsfig{file=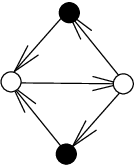,width=0.1\linewidth}}& \epsfig{file=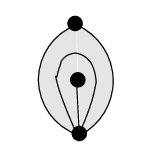,width=0.15\linewidth}\\
\hline
\raisebox{9mm}{${\rm{V}}$}&  \raisebox{3mm}{\epsfig{file=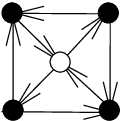,width=0.1\linewidth}}& \epsfig{file=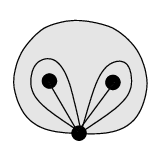,width=0.15\linewidth}\\
\hline
\end{tabular}
&
\begin{tabular}{|c|c|c|}
\hline
Type& \phantom{x} Diagram \phantom{x}  & Orbifold\\
\hline 
\raisebox{9mm}{${\widetilde{ \rm{III}}a}$}&  \raisebox{9mm}{\psfrag{2}{\scriptsize $2$}\epsfig{file=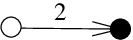,width=0.1\linewidth}}& \epsfig{file=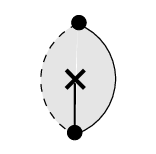,width=0.15\linewidth}\\
\hline
\raisebox{9mm}{${\widetilde{ \rm{III}}b}$}&   \raisebox{9mm}{\psfrag{2}{\scriptsize $2$}\epsfig{file=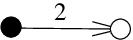,width=0.1\linewidth}}&\epsfig{file=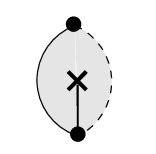,width=0.15\linewidth}\\
\hline
\raisebox{9mm}{${\widetilde{\rm{IV}}}$}&   \raisebox{5mm}{\psfrag{2}{\scriptsize $2$}\epsfig{file=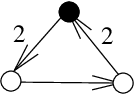,width=0.1\linewidth}}&\epsfig{file=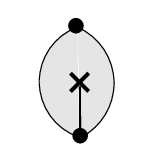,width=0.15\linewidth}\\
\hline
\raisebox{9mm}{${\widetilde{\rm{V}}_1}$}&  \raisebox{1mm}{\psfrag{2}{\scriptsize $2$}\epsfig{file=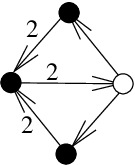,width=0.1\linewidth}} &\epsfig{file=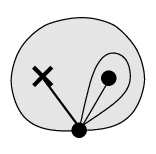,width=0.15\linewidth}\\
\hline
\raisebox{9mm}{${\widetilde{\rm{V}}_2}$}&  \raisebox{1mm}{\psfrag{2}{\scriptsize $2$}\epsfig{file=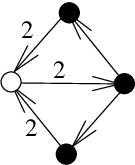,width=0.1\linewidth}} &\epsfig{file=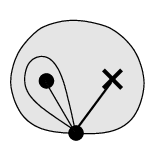,width=0.15\linewidth}\\
\hline
\raisebox{9mm}{${\widetilde{\rm{V}}_{12}}$}&   \raisebox{5mm}{\psfrag{2}{\scriptsize $2$} \psfrag{4}{\scriptsize $4$}\epsfig{file=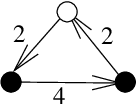,width=0.1\linewidth}}&\epsfig{file=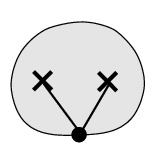,width=0.15\linewidth}\\
\hline
\end{tabular}
\\
\\
\end{tabular}
\end{center}
\end{table}

\begin{remark}
There are also several exceptional blocks which have no white vertices and are used only to represent some triangulations of
small exceptional orbifolds, namely, sphere with four marked points (some of which are punctures and some are orbifold points). See~\cite[Table~3.2]{FeSTu3} for the list. 

\end{remark}

Block-decomposable diagrams are in one-to-one correspondence with adjacency matrices of arcs of ideal (tagged) triangulations of bordered two-dimensional surfaces and orbifolds with marked points (see~\cite[Section~13]{FST} and~\cite{FeSTu3} for the detailed explanations). Mutations of block-decomposable diagrams correspond to flips of (tagged) triangulations. In particular, this implies that mutation class of any block-decomposable diagram is finite, and any subdiagram of a block-decomposable diagram is block-decomposable too.

Theorem~\ref{class} shows that block-decomposable diagrams almost exhaust mutation-finite ones. In skew-symmetric case this implies the following easy corollary:
%It is proved in~\cite{FeSTu1} and~\cite{FeSTu2}  that block-decomposable diagrams almost exhaust mutation-finite ones. Namely, any mutation-finite non-decomposable  diagram of order at least $3$ is mutation-equivalent to one of $11+7$ exceptional diagrams, see~\ref{class}. In skew-symmetric case this implies in addition that all small diagrams are block-decomposable:

\begin{prop}[\cite{FeSTu1},Theorem~5.11]
\label{small}
Any skew-symmetric mutation-finite diagram of order less than $6$ is block-decomposable.
\end{prop}

We will use the surface and orbifold presentations of block-decomposable diagrams of finite and affine type, see Table~\ref{surface realizations}.

\begin{table}[!h]
\begin{center}
\caption{Surfaces and orbifolds corresponding to block-decomposable diagrams of finite and affine type}
\label{surface realizations}
\begin{tabular}{|l|l|}
\hline
\vphantom{$\int^0$}$A_n$ & disk\\
\vphantom{$\int^0$}$B_n=C_n$ & disk with an orbifold point\\
\vphantom{$\int^0_0$}$D_n$ & disk with a puncture\\
\hline 
\vphantom{$\int\limits^.$}$\widetilde A_{n}$& annulus\\
\vphantom{$\int^0$}$\widetilde B_n$&  disk with a puncture and an orbifold point\\
\vphantom{$\int^0$}$\widetilde C_n$&  disk with two orbifold points\\
\vphantom{$\int^0_0$}$\widetilde D_n$&  disk with two punctures\\
\hline
\end{tabular}
\end{center}
\end{table}

\begin{remark}
A mutation class $\widetilde A_{k,n-k}$ (of affine type $\widetilde A_{n-1}$) corresponds to an annulus with $k$ marked points on one boundary component and $n-k$ on the other.

\end{remark}

\section{Subdiagrams of mutation-finite diagrams}
\label{subd}

In this section, we list some technical facts we are going to use in the sequel.

\subsection{Double arrows in diagrams of mutation classes $\widetilde A_n$, $\widetilde B_n$, $\widetilde C_n$ and $\widetilde D_n$ }
By a {\it double arrow} we mean an arrow labeled by $4$ (the origin of this notation is in the presentation of skew-symmetric diagrams by quivers).
A double arrow in a decomposable diagram may arise in two ways: either it is contained in the block ${\widetilde{ \rm{V}}_{12}}$ or it is glued of two simple arrows from two blocks.
Since the blocks correspond to some pieces of a surface/orbifold, there are restrictions on some arrangements of blocks in block decompositions of diagrams of a given mutation type.
\begin{itemize}
\item Block of type ${\rm{IV}}$, as well as a combination of blocks of type $\widetilde{\rm{IV}}$ or ${\rm{IV}}$ with a block of type ${\rm{I}}$ or ${\rm{II}}$ leading to a double arrow, results in a puncture on the corresponding surface/orbifold, so all these do not appear in diagrams of type $\widetilde A_n$ and $\widetilde C_n$. 
\item Gluing of two blocks of types ${\rm{I}}$ or ${\rm{II}}$ leading to a double arrow results in an annulus with one marked point at each boundary component. There is no way to glue any blocks to this annulus to obtain a closed disk with at most two punctures or orbifold points in total. Thus, these do not appear in diagrams of type $\widetilde B_n$, $\widetilde C_n$ and $\widetilde D_n$. 
\item Gluing of two blocks of types $\widetilde{\rm{IV}}$ or ${\rm{IV}}$ results in a closed sphere with punctures and/or orbifold points, so these do not appear in affine diagrams.

\end{itemize}

Based on the restrictions above, we list all possible ways to get a double arrow inside decomposable affine diagrams in Table~\ref{double edges}.
Taking into account the fact that both block of type I and block of type II correspond to a triangle on a surface/orbifold (with only difference that for the former the triangle has a boundary arc),
we also write a {\it reduced list} of the possibilities, where we exclude blocks of type I.

\begin{table}[!h]
\begin{center}
\caption{Possibilities for double arrows in decomposable affine diagrams}
\label{double edges}
\begin{tabular}{|l|l|l|}
\hline
 type& block decompositions& reduced list of decompositions\\
\hline
\vphantom{$\int\limits^.$}$\widetilde A_n$&  I+II, II+II& II+II\\
\hline
\vphantom{$\int\limits^.$}\raisebox{-0mm}{$\widetilde B_n$}&   I+$\widetilde{\rm{IV}}$, II+$\widetilde{\rm{IV}}$&  II+$\widetilde{\rm{IV}}$\\
\hline
\vphantom{$\int\limits^.$}\raisebox{-0mm}{$\widetilde C_n$}&   $\widetilde{\rm{V}}_{12}$& $\widetilde{\rm{V}}_{12}$\\
\hline
\vphantom{$\int\limits^.$}$\widetilde D_n$&  I+IV, II+IV&  II+IV\\
\hline
\end{tabular}
\end{center}
\end{table}

\subsection{Oriented cycles in mutation-finite diagrams}

\begin{lemma}
\label{cycles}
Let $\P$ be an oriented chordless cycle, $\P\subset \D$, where $\D$ is a mutation-finite diagram. Then $\P$ is either composed of simple arrows or it coincides with one of the cycles in Table.~\ref{short cycles}

\end{lemma}

\begin{proof}
First, suppose that $\P$ is block-decomposable. It is easy to see that either $\P$ is a block or $\P$ is composed of blocks having at least two white vertices. Considering these two cases we get diagrams 1-7 in  Table.~\ref{short cycles}
(to simplify the reasoning we note that a block of type IV never lies in a block decomposition of an oriented cycle, so all decomposable cycles different from blocks are glued of blocks of types I, II and {$\widetilde{\rm{IV}}$}).

\begin{table}
\begin{center}
\label{short cycles}
\caption{Mutation-finite oriented cycles with a non-simple arrow }
\begin{tabular}{|c|c|c|c|}
\hline
&diagram& mutation class& triangulation (if any)\\
\hline
\raisebox{5mm}{1}& 
\psfrag{k}{\scriptsize }
\psfrag{l}{\scriptsize $$}
\psfrag{m}{\scriptsize $4$}
\epsfig{file=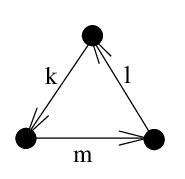,width=0.09\linewidth}
& \raisebox{5mm}{$\widetilde A_{2,1}$}& \epsfig{file=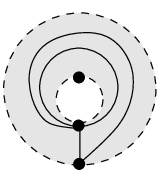,width=0.08\linewidth}\\
\hline
\raisebox{5mm}{2}& 
\psfrag{k}{\scriptsize $2$}
\psfrag{l}{\scriptsize $2$}
\psfrag{m}{\scriptsize }
\epsfig{file=./pic/d3.eps,width=0.09\linewidth}
& \raisebox{5mm}{$B_3$}& \epsfig{file=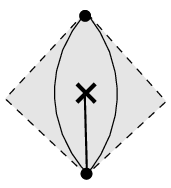,width=0.08\linewidth}\\
\hline
\raisebox{5mm}{3}& 
\psfrag{k}{\scriptsize $2$}
\psfrag{l}{\scriptsize $2$}
\psfrag{m}{\scriptsize $4$}
\epsfig{file=./pic/d3.eps,width=0.09\linewidth}
& \raisebox{5mm}{
\begin{tabular}{c}
$\widetilde B_2$\\
(see Remark~\ref{b2})
\end{tabular}
}& \epsfig{file=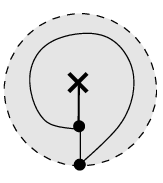,width=0.08\linewidth}\\
\hline
\raisebox{5mm}{4}& 
\psfrag{k}{\scriptsize $2$}
\psfrag{l}{\scriptsize $2$}
\psfrag{m}{\scriptsize $4$}
\epsfig{file=./pic/d3.eps,width=0.09\linewidth}
& \raisebox{5mm}{$\widetilde C_2$}& \epsfig{file=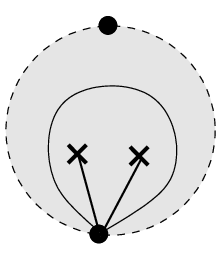,width=0.08\linewidth}\\
\hline
\raisebox{7mm}{5}& 
\psfrag{k}{\scriptsize $4$}
\psfrag{l}{\scriptsize $4$}
\psfrag{m}{\scriptsize $4$}
\raisebox{1mm}{\epsfig{file=./pic/d3.eps,width=0.09\linewidth}}
& 
\raisebox{7mm}{\begin{tabular}{c}
punctured \\
torus\\
\end{tabular}}
& \epsfig{file=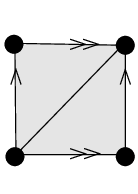,width=0.07\linewidth}\\
\hline
\raisebox{5mm}{6}& 
\psfrag{k}{\scriptsize $2$}
\psfrag{l}{\scriptsize $2$}
\psfrag{m}{\scriptsize }
\psfrag{n}{\scriptsize }
\raisebox{-1mm}{\epsfig{file=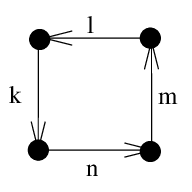,width=0.09\linewidth}}
& \raisebox{5mm}{$\widetilde B_3$}& \epsfig{file=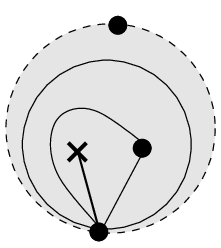,width=0.08\linewidth}\\
\hline
\raisebox{5mm}{7}& 
\psfrag{k}{\scriptsize $2$}
\psfrag{l}{\scriptsize $2$}
\psfrag{m}{\scriptsize $2$}
\psfrag{n}{\scriptsize $2$}
\raisebox{-2mm}{\epsfig{file=./pic/d4.eps,width=0.09\linewidth}}
&
\raisebox{5mm}{\begin{tabular}{c}
sphere with \\2 punctures and\\
2 orbifold points\\ 
\end{tabular}}
& \raisebox{-0.99mm}{\epsfig{file=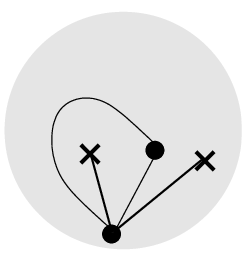,width=0.08\linewidth}}\\
\hline
\hline
\raisebox{5mm}{8}&
\psfrag{k}{\scriptsize $2$}
\psfrag{l}{\scriptsize }
\psfrag{m}{\scriptsize $2$}
\psfrag{n}{\scriptsize }
\raisebox{-0.99mm}{\epsfig{file=./pic/d4.eps,width=0.09\linewidth}}& \raisebox{5mm}{$F_4$}&\\
\hline
\raisebox{5mm}{9}&
\psfrag{2}{\scriptsize $2$}
\epsfig{file=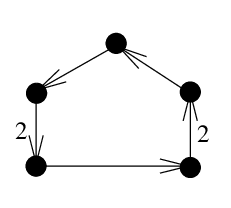,width=0.09\linewidth}& \raisebox{5mm}{$\widetilde F_4$}&\\
\hline
\raisebox{5mm}{10}&
\psfrag{2}{\scriptsize $2$}
\raisebox{-0.99mm}{\epsfig{file=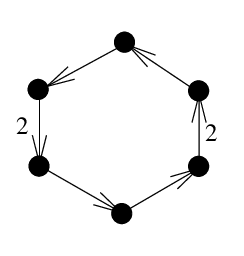,width=0.09\linewidth}}& \raisebox{5mm}{$F_4^{(*,+)}$}&\\
\hline
\raisebox{5mm}{11}&
\psfrag{k}{\scriptsize $3$}
\psfrag{l}{\scriptsize $3$}
\psfrag{m}{\scriptsize }
\epsfig{file=./pic/d3.eps,width=0.09\linewidth}& \raisebox{5mm}{$\widetilde G_2$}&\\
\hline
\raisebox{5mm}{12}&
\psfrag{k}{\scriptsize $3$}
\psfrag{l}{\scriptsize $3$}
\psfrag{m}{\scriptsize $4$}
\epsfig{file=./pic/d3.eps,width=0.09\linewidth}& \raisebox{5mm}{$\widetilde G_2$}&\\
\hline
\raisebox{5mm}{13}&
\psfrag{k}{\scriptsize $3$}
\psfrag{l}{\scriptsize }
\psfrag{m}{\scriptsize $3$}
\psfrag{n}{\scriptsize }
\raisebox{-0.99mm}{\epsfig{file=./pic/d4.eps,width=0.09\linewidth}}& \raisebox{5mm}{$G_2^{(*,+)}$}&\\
\hline
\end{tabular}
\end{center}
\end{table}

Suppose now that $\P$ is not decomposable. We consider the cases when $\P$ is skew-symmetric and non-skew-symmetric separately. 

If $\P$ is skew-symmetric then any arrow of $\P$ is labeled by $1$ or $4$. Furthermore, being a non-decomposable skew-symmetric mutation-finite diagram, $\P$ has at least 6 vertices (see Proposition~\ref{small}).
Suppose that one of the arrows is labeled by $4$ (otherwise there is nothing to prove). Then this arrow together with its two neighbors builds one of the subdiagrams in Fig.~\ref{4}. However, all of the four diagrams are mutation-infinite, which contradicts the assumptions.

\begin{figure}[!h]
\begin{center}
\psfrag{4}{\scriptsize $4$}
\epsfig{file=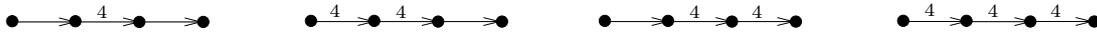,width=0.9\linewidth}
\caption{Four mutation-infinite diagrams.}
\label{4}
\end{center}
\end{figure}

If $\P$ is non-skew-symmetric non-decomposable diagram, then by Theorem~\ref{class} it is mutation-equivalent to one of the diagrams in the bottom part of Table~\ref{exceptional}. Using~\cite{Kel}, we check the mutation classes of these diagrams for cyclic diagrams and list all of them in rows  8--13  of  Table~\ref{short cycles}. (In fact, the mutation classes of these diagrams are not too big, at most 90 diagrams according to~\cite{Kel}, most of the diagrams having more arrows than the cyclic one should have).

\end{proof}

\begin{remark}
\label{b2}
The orbifold in row $3$ of Table~\ref{short cycles} is a disk with one puncture, one orbifold point and one marked point at the boundary, thus it may be considered as a partial case of mutation type $\t B_n$ (see Table~\ref{surface realizations}). This is the reason we call it $\t B_2$ (at the same time, the corresponding diagram is mutation-equivalent to $\t C_2$, see the next row in the table).   
\end{remark}

The following is an immediate corollary of Lemma~\ref{cycles}.

\begin{cor}
\label{cycles-af}
Let $\P$ be an oriented cycle. If $\P$ is a subdiagram of an affine diagram and not a subdiagram of any finite diagram 
then $\P$ is of one of the six types listed in Table~\ref{cycles-t}

\end{cor}

\begin{center}
\begin{table}
\caption{Oriented cycles of non-finite type in affine diagrams}
\label{cycles-t}
\begin{tabular}{|c|c|c|c|c|c|}
\hline
\psfrag{k}{\scriptsize }
\psfrag{l}{\scriptsize $$}
\psfrag{m}{\scriptsize $4$}
\epsfig{file=./pic/d3.eps,width=0.1\linewidth}
&
%\psfrag{k}{\scriptsize $2$}
%\psfrag{l}{\scriptsize $2$}
%\psfrag{m}{\scriptsize $4$}
%\epsfig{file=./pic/d3.eps,width=0.1\linewidth}
%&
\psfrag{k}{\scriptsize $2$}
\psfrag{l}{\scriptsize $2$}
\psfrag{m}{\scriptsize $4$}
\epsfig{file=./pic/d3.eps,width=0.1\linewidth}
&
\psfrag{k}{\scriptsize $2$}
\psfrag{l}{\scriptsize $2$}
\psfrag{m}{\scriptsize }
\psfrag{n}{\scriptsize }
\epsfig{file=./pic/d4.eps,width=0.1\linewidth}
&
\psfrag{2}{\scriptsize $2$}
\epsfig{file=./pic/d5.eps,width=0.1\linewidth}
&
\psfrag{k}{\scriptsize $3$}
\psfrag{l}{\scriptsize $3$}
\psfrag{m}{\scriptsize $$}
\epsfig{file=./pic/d3.eps,width=0.1\linewidth}
&
\psfrag{k}{\scriptsize $3$}
\psfrag{l}{\scriptsize $3$}
\psfrag{m}{\scriptsize $4$}
\epsfig{file=./pic/d3.eps,width=0.1\linewidth}
\\
\hline
\vphantom{$\int\limits^.$}$\widetilde A_{2,1}$&$\widetilde B_2$ or $\widetilde C_2$&$\widetilde B_3$&$\widetilde F_4$&$\widetilde G_2$ &$\widetilde G_2$ \\
\hline
\end{tabular}
\end{table}
\end{center}

\subsection{Non-oriented cycles in mutation-finite diagrams}

We will also use the following lemma proved by Seven in~\cite{S1}.

\begin{lemma}[Proposition~2.1 (iv),~\cite{S1}]
\label{non-or}
Let $\D$ be a simply-laced mutation-finite skew-symmetric diagram and
let $\C\subset \D$ be a non-oriented chordless cycle.
Then for each vertex $x\in \D$ the number of arrows connecting $x$ with $\C$ is even.

\end{lemma}

\section{Groups defined by diagrams of affine type}
\label{aff group def}

\label{groups}
In this section, we define a group associated to a diagram of affine type. Our definition is similar to one given by Barot and Marsh~\cite{BM} for finite type, but with additional relations for some affine subdiagrams, see Table~\ref{add-t}.

Let $\G$ be a diagram with $n+1$ vertices. Following~\cite{BM} (and~\cite{S}), define
$$
m_{ij}=
\begin{cases}
2 & \text{if $i$ and $j$ are not joined;} \\
3 & \text{if $i$ and $j$ are joined by an arrow labeled by $1$;} \\
4 & \text{if $i$ and $j$ are joined by an arrow labeled by $2$;} \\
6 & \text{if $i$ and $j$ are joined by an arrow labeled by $3$.} 
\end{cases}
$$

\begin{defin}[Group $W_\G$ for a diagram $\G$ of affine type]
\label{def gp}
The group $W_{\G}$  with generators $s_1,\dots,s_{n+1}$ is defined by the following relations of four types:

\begin{itemize}
\item[(R1)] $s_i^2=e$ for all $i=1,\dots,n+1$;

\item[(R2)] $(s_is_j)^{m_{ij}}=e$ for all $i,j$ not joined by an arrow labeled by $4$;

\item[(R3)] (cycle relations) for every chordless oriented cycle $\C$ of length $d$ given by 
$$i_0\stackrel{w_{i_0i_1}}{\to} i_1\stackrel{w_{i_1i_2}}{\to}\cdots\stackrel{w_{i_{d-2}i_{d-1}}}{\to} i_{d-1}\stackrel{w_{i_{d-1}i_0}}{\to}i_0$$
and for every $l=0,\dots,d-1$ we define a number 
$$t(l)=\left(\prod\limits_{j=l}^{l+d-2}\!\!\!\!\sqrt{w_{i_ji_{j+1}}}\ -\sqrt{w_{i_{l+d-1}i_l}}\right)^2,$$ 
where the indices are considered modulo $d$;
now for every $l$ such that $t(l)<4$ we take relations
$$
(s_{i_l}s_{i_{l+1}}\dots s_{i_{l+d-2}}s_{i_{l+d-1}}s_{i_{l+d-2}}\dots s_{i_{l+1}})^{m(l)}=e,
$$
where 
$$
m(l)=
\begin{cases}
2 & \text{if $t(l)=0$;} \\
3 & \text{if $t(l)=1$;} \\
4 & \text{if $t(l)=2$;} \\
6 & \text{if $t(l)=3$.} 
\end{cases}
$$

\item[(R4)] (additional affine relations) for every subdiagram of $\G$ of the form shown in the first column of Table~\ref{add-t} we take the relations listed in the second column of the table.

\end{itemize}
\end{defin}

\begin{remark}
The fact that $t(l)$ in (R3) is integer follows from skew-symmetrizability of a matrix associated to $\G$, i.e. the product of weights along any chordless cycle of $\G$ is a perfect square (see Section~\ref{dm}).
\end{remark}

\begin{remark}
In the sequel by a {\it cycle} we always mean a chordless cycle. We will also refer to the relations above as {\it relation of type} (R1) (respectively, (R2), (R3) or (R4)).  
\end{remark}

\begin{table}
\caption{Additional relations for subdiagrams of affine diagrams. The type of the corresponding subdiagram is shown in the third column.}
\label{add-t}
\begin{tabular}{|c|c|c|}
\hline
\vphantom{$\int\limits^.$}Subdiagram & Relations & Type\\
\hline
\psfrag{4_}{\scriptsize $4$}
\psfrag{2_}{\scriptsize $$}
\psfrag{1}{\scriptsize $1$}
\psfrag{2}{\scriptsize $2$}
\psfrag{3}{\scriptsize $3$}
\psfrag{4}{\scriptsize $4$}
\epsfig{file=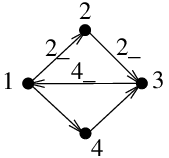,width=0.13\linewidth}
&
\raisebox{9mm}{$(s_1s_2s_3s_4s_3s_2)^2=e$}
&\raisebox{9mm}{\begin{tabular}{c}$\t A_{3}$\\ ($\vphantom{\int\limits^0}\t A_{2,2}$)\end{tabular}}\\
\hline&&\\
\psfrag{u}{\scriptsize $1$}
\psfrag{v}{\scriptsize $2$}
\psfrag{1}{\scriptsize $3$}
\psfrag{2}{\scriptsize $4$}
\psfrag{3}{\scriptsize $5$}
\psfrag{4}{\scriptsize $6$}
\psfrag{n}{\scriptsize $n+1$}
\psfrag{n1}{\scriptsize $n$}
\epsfig{file=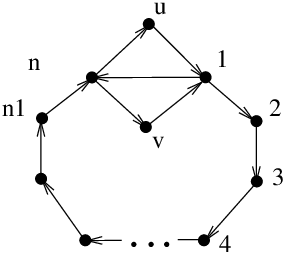,width=0.2\linewidth}
& \raisebox{13mm}{$(s_1s_2s_3s_2s_1 \ s_4s_5\dots s_{n}s_{n+1}s_{n}\dots s_5s_4)^2=e $}&\raisebox{13mm}{$\t D_n$,\ $n\ge 4$} \\
\hline
\psfrag{4_}{\scriptsize $4$}
\psfrag{2_}{\scriptsize $2$}
\psfrag{1}{\scriptsize $1$}
\psfrag{2}{\scriptsize $2$}
\psfrag{3}{\scriptsize $3$}
\psfrag{4}{\scriptsize $4$}
\epsfig{file=./pic/rel_B3.eps,width=0.13\linewidth}
& 
\raisebox{9mm}{%\begin{tabular}{l}
$(s_2s_3s_4s_1s_4s_3)^2=e$ %$(s_2s_1s_4s_3s_4s_1)^2=e$ 
%\end{tabular}
}
&\raisebox{9mm}{$\t B_3$}\\
\hline&&\\
\psfrag{u}{\scriptsize $n+1$}
\psfrag{2_}{\scriptsize $2$}
\psfrag{1}{\scriptsize $1$}
\psfrag{2}{\scriptsize $2$}
\psfrag{3}{\scriptsize $3$}
\psfrag{4}{\scriptsize $4$}
\psfrag{n}{\scriptsize $n-1$}
\psfrag{n+1}{\scriptsize $n$}
\epsfig{file=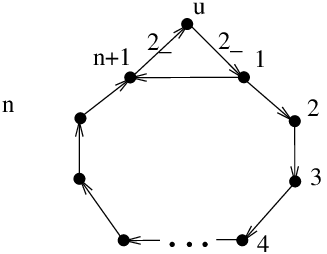,width=0.21\linewidth}\phantom{wv}
& \raisebox{13mm}{$(s_{n+1}s_1s_{n+1} \ s_2s_3\dots s_{n-1}s_{n}s_{n-1}\dots s_3s_2)^2=e $}&\raisebox{13mm}{$\t B_n$,\ $n\ge 3$} \\
\hline&&\\
\psfrag{1}{\scriptsize $1$}
\psfrag{2}{\scriptsize $2$}
\psfrag{3_}{\scriptsize $3$}
\psfrag{3}{\scriptsize $3$}
\psfrag{4}{\scriptsize $4$}
\epsfig{file=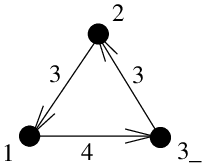,width=0.11\linewidth}
&
\raisebox{7mm}{%\begin{tabular}{c}
$(s_2s_1s_2s_1s_2s_3)^2=e $
%$(s_2s_3s_2s_3s_2s_1)^2=e $
%\end{tabular}
}
&\raisebox{7mm}{$\t G_2$}\\
\hline
\end{tabular}
\end{table}

\begin{remark}
Table~\ref{subdiagrams} shows which of the subdiagrams from Table~\ref{add-t} appear in affine diagrams depending on the type of the latter. Note that there are two distinct additional affine relations for affine type $\t B_3$.

\end{remark}

\begin{table}
\caption{The types of affine diagrams (left column) containing affine subdiagrams requiring additional relations (right column)}
\label{subdiagrams}
\begin{tabular}{|c|c|}
\hline
\begin{tabular}{c}\vphantom{$\int\limits^.$}Mutation types of\\ affine diagrams\end{tabular} & Subdiagrams appearing in Table~\ref{add-t}\\
\hline
\vphantom{$\int\limits^.$}$\t A_{n,1}$, \ $n\ge 1$& \\
\hline
\vphantom{$\int\limits^.$}$\t A_{p,q}$, \ $p,q\ge 2$ & $\t A_{2,2}$\\
\hline
\vphantom{$\int\limits^.$}$\t D_n$, \ $n\ge 4$ & $\t A_{2,2}$, $\t D_k$, \ $k\le n$\\
\hline
\vphantom{$\int\limits^.$}$\t E_6$ &  $\t A_{2,2}$, $\t D_k$, $k\le 5$\\
\hline
\vphantom{$\int\limits^.$}$\t E_7$ &  $\t A_{2,2}$, $\t D_k$, $k\le 6$\\
\hline
\vphantom{$\int\limits^.$}$\t E_8$ &  $\t A_{2,2}$, $\t D_k$, $k\le 7$\\
\hline
\vphantom{$\int\limits^.$}$\t B_3$             & $\t B_3$\\
\hline
\vphantom{$\int\limits^.$}$\t B_n$, \ $n\ge 4$ & $\t A_{2,2}$, $\t B_k$, $k\le n$\\
\hline
\vphantom{$\int\limits^.$}$\t C_n$, \ $n\ge 2$ &                           \\
\hline
\vphantom{$\int\limits^.$}$\t F_4$ & $\t B_3$\\  
\hline
\vphantom{$\int\limits^.$}$\t G_2$ & $\t G_2$\\
\hline
\end{tabular}
\end{table}

The relations of types (R1), (R2) and (R3) are the relations introduced by Barot and Marsh~\cite{BM} for 
diagrams of finite type.  The expression for $t(l)$ (and $m(l)$) was suggested by Seven~\cite{S}. It is easy to see that in the case of finite Weyl groups the number $t(l)$ is either $0$ or $1$, and the expression for $m(l)$ coincides with one from~\cite{BM}.
After adding relations of type (R4) our definition still coincides with the definition 
in~\cite{BM} when restricted to diagrams of finite type since the diagrams used in relations of type (R4) are of affine type and cannot be subdiagrams of diagrams of finite type.

\begin{theorem}[\cite{BM}, Theorem A]
\label{bm}
Let $G_0$ be a finite Weyl group, and let $\G$ be a diagram of the same type as $G_0$. Then $G_0$ is isomorphic to $W_{\G}$. 

\end{theorem}

\begin{lemma}[Seven~\cite{S}, Theorem 1.1]
\label{seven}
Let $W_0$ be an affine Weyl group, and let $\G$ be a diagram of the same type as $W_0$. Then $W_0$ is isomorphic to a quotient group of $W_{\G}$. 

\end{lemma}

%In constrast to the finite case, an affine Weyl group may not be isomorphic to $\t W_{\G}$, see 
%Section~\ref{examples}.
% the next section. 

%\begin{remark}
%\label{an1}
%There are many (mutationally non-equivalent) acyclic diagrams corresponding to affine Weyl group $\t A_n$, namely, such a diagram (denote it by $\t A_n^k$ is a non-oriented cycle with $k$ consequent arrows directed clockwise and $n-k+1$ arrows directed counter-clockwise, we may assume $1\le k\le n/2$ up to reversing all the arrows. These diagrams are not mutation-equivalent for distinct $k$. The diagram of type $\t A_3$ shown in Table~\ref{add-t} is mutation-equivalent to $\t A_3^2$, and for any diagram $\G$ mutation-equivalent to $\t A_3^1$ the group $\t W_{\G}$ is isomorphic to $\t A_3$. The same holds for larger $n$: no diagram mutation-equivalent to $\t A_n^1$ contains any diagram from Table~\ref{add-t} as a subdiagram, however, every diagram mutation-equivalent to   $\t A_n^k$ for $k>1$ contains one.      

%\end{remark}

In Section~\ref{invariance} we prove the invariance of the group $W_{\G}$ under the mutation in the case of affine diagrams: 

\begin{theorem}
\label{thm-aff}
Let $W$ be an affine Weyl group and let $\G$ be a diagram mutation-equivalent to an orientation of a Dynkin diagram of the same type as $W$ different from an oriented cycle. Then $W$ is isomorphic to $W_{\G}$. 

\end{theorem}

In particular, Theorem~\ref{thm-aff} implies that all groups $W_{\G}$ obtained for the affine diagrams are Coxeter groups.
 
Denote by  $\t W_{\G}$  the group obtained from  $W_{\G}$ by omitting all additional affine relations.

As it is shown in Section~\ref{difference}, the relations of type (R4) are essential: for some diagram in the mutation class of $\G$ the group  $\t W_{\G}$ is not isomorphic to $W$.

\begin{remark}
\label{cn}
As one can see from Table~\ref{subdiagrams}, the diagrams  mutation-equivalent to ones of type $\t C_n$ do not contain any subdiagram from Table~\ref{add-t}. This implies that for $\G$ of the type $\t C_n$ the groups $\t W_{\G}$ and $W_{\G}$ are isomorphic, and thus  $W_{\G}$ is isomorphic to $W$. The same holds for diagrams of type $\t A_{k,1}$.

\end{remark}

\section{Symmetry and redundancy of relations in the presentation of $W_{\G}$}
\label{redundance}

In this section, we show that the additional affine relations in the definition of the group  $W_{\G}$ 
imply more similar relations (obtained from symmetries of the diagram $\G$)
and that the number of cycle relations (type (R3) relations) in the presentation of $W_{\G}$ can be decreased significantly.
%More precisely, we show that for each pseudocycle in $\G$ it is sufficient to take one suitable relation.
These properties will be extensively used later while proving the invariance of $W_{\G}$ under mutations. 
%Namely, proving the equivalence of the groups $W_\G$ and $W_{\mu_i(\G)}$ it is convenient to derive {\it defining} relations for one group from {\it all} relations of the other.

\subsection{Symmetries}

\begin{lemma}
\label{reversion}
Let $\G$ be a diagram of finite or affine type and $\G^{op}$ be the same diagram with all the directions of arrows reversed.
Then the groups $W_{\G^{op}}$ and $W_\G$ are isomorphic.

%Additional affine relations  (see Table~\ref{add-t}) are invariant under reversion of all arrors in the diagram.

\end{lemma}

\begin{proof}
Note that the subdiagrams that are supports of the relations of types (R1)--(R4) are the same for $\G$ and $\G^{op}$. Thus, 
it is sufficient to prove the statement for each subdiagram supporting a relation of $W_{\G^{op}}$ and $W_\G$.
In particular, it is clear that the relations of types (R1) and (R2) do not depend on the directions of arrows.

Our aim is to prove that the relations of types (R3) and (R4) do not depend on the simultaneous change of orientation of all arrows.
First, suppose that $\G$ is not a diagram defining additional affine relation of type $\widetilde B_3$ or $\widetilde G_2$. 
Then all relations of types (R3) and (R4) have form
$$
(s_iws_jw^{-1})^k=e
$$
(where $w$ is a word in the alphabet $\{s_1,\dots,s_{n+1}\}$), and
after simultaneous reversing of all arrows the corresponding relation rewrites as 
$$
(s_jw^{-1}s_iw)^k=e.
$$
The latter is clearly conjugate to the initial relation:
$$
(s_jw^{-1}s_iw)^k=s_jw^{-1}(s_iws_jw^{-1})^k ws_j,
$$
so these relations are equivalent.
 
It remains to check the statement for the diagrams defining additional affine relations of type $\widetilde B_3$ or $\widetilde G_2$. But in these cases reversing of all arrows does not affect additional affine relations (since we include both directions in the definition of the group) and all the other relations are treated as above.
\end{proof}

\begin{remark}
Lemma~\ref{reversion} for diagrams of finite type was proved in~\cite{BM} (see Prop.~4.6). 
\end{remark}

\begin{remark}
The diagram of type $\widetilde D_k$ defining additional affine relation has extra symmetry interchanging the vertices $1$ and $2$ (see Table~\ref{add-t}).
The relation obtained via this symmetry may be obtained from the initial one by interchanging $s_1$ and $s_2$
($s_1$ and $s_2$ commute, and they are neighbors in the relation).

Similarly, there is a symmetry in the  diagram defining additional affine relation  of type $\widetilde A_{2,2}$
(swapping vertices 2 and 4). After application of this symmetry the relation rewrites as $(s_1s_4s_3s_2s_3s_4)^2=e$, which is equivalent 
to the initial relation  $(s_1s_2s_3s_4s_3s_2)^2=e$: 
$$
 (s_1s_4s_3s_2s_3s_4)^2 \stackrel{(s_2s_3)^3=e}{=}(s_1s_4s_2s_3s_2s_4)^2 \stackrel{(s_2s_4)^2=e}{=}
(s_1s_2s_4s_3s_4s_2)^2 \stackrel{(s_3s_4)^3=e}{=}(s_1s_2s_3s_4s_3s_2)^2.
$$

\end{remark}

\begin{remark}
One can see that the additional relation for $\t B_3$ is equivalent to $(s_2 s_1 s_4 s_3 s_4 s_1)^ 2 = e$, and the additional relation for $\t G_2$ is equivalent to $(s_2 s_3 s_2 s_3 s_2 s_1)^ 2 = e$ (we used {\tt Magma}~\cite{BCP} for verification of the equivalence).
\end{remark}

\subsection{Redundancy of cycle relations}

By the definition of the group  $W_{\G}$, each oriented cycle of order $k$ defines $k$ relations of type (R3). In fact, not all of them 
are essential: in most cases it is sufficient to choose just one suitable relation.

\begin{lemma}[Lemma~4.1,~\cite{BM}]
Let $\C$ be an oriented simply-laced cycle. Then all relations of the group $W_\C$ follow from relations
of types (R1), (R2) and any one relation of type (R3).

\end{lemma} 

%Already in the case of an oriented cycle of type $B_3$ (number 3 in Table~\ref{short cycles}) it is not clear that each of the three possible cycles may be chosen as a cycle for a defining relation (see~\cite{BM}, the paragraph between Lemma~4.2 and Example~4.3).
For non-simply-laced cycles the situation is more involved: already for an oriented cycle of type $B_3$ it is not clear whether each of the three relations can be chosen as a defining relation (see~\cite{BM}). However, it is shown in~\cite[Lemmas~4.2 and~4.4]{BM} that for oriented cycles of mutation types $B_3$ and $F_4$ (rows 2 and 8 in Table~\ref{short cycles}) one can choose any cycle relation with $m(l)=2$, and all the other cycle relations for the given cycle will follow from the chosen one.

The results of Lemmas~4.1, 4.2 and 4.4 from~\cite{BM} may be summarized as follows.

\begin{lemma}[Lemmas~4.1, 4.2 and 4.4,~\cite{BM}]
Let  $\C$ be an oriented cycle of finite type. 
Then there exists a cycle relation $r_l$ for $\C$ such that $m(l)=2$ (see Definition~\ref{def gp}).
Moreover, $r_l$ implies all other cycle relations supported by the cycle $\C$.

\end{lemma}

A direct computation (very similar to one in~\cite{BM}) shows that the statement above can be extended to almost all affine 
cyclic diagrams:

\begin{lemma}
Let  $\C$ be an oriented cycle of affine type not requiring additional affine relations
(i.e. distinct from the cycle of type $\widetilde G_2$ in Table~\ref{add-t}).

If $\C$ is not of type $\widetilde A_{2,1}$ then  there exists a cycle relation $r_l$ for $\C$ such that $m(l)=2$ (see Definition~\ref{def gp}), and this relation $r_l$ implies all the other cycle relations supported by the cycle $\C$.

If $\C$ is of type $\widetilde A_{2,1}$ then $r(l)=3$ for all $l$ and all the three relations are equivalent.  

In addition, if $\C$ is an oriented cycle  of type $\widetilde G_2$ shown in Table~\ref{add-t},
then two of the cycle relations (with $m(l)=6$) follow from the additional relation and the third cycle relation
(with $m(l)=3$).

\end{lemma}

\begin{cor}
Let  $\C$ be an oriented cycle of affine type not requiring additional affine relations.
Then there exists one cycle relation for $\C$ implying all the other cycle relations for $\C$.

\end{cor}

\begin{remark}
It is possible to prove that the additional relation for $\t G_2$ together with relations of type (R1) and (R2) do not form defining set of relations for this diagram without the cycle relation with $m(l)=3$ (indeed, the latter relation contains odd number of letters $s_3$ while all the other relations of types (R1)--(R4) have even number of $s_3$'s).

\end{remark}

\section{Proof of Theorem~\ref{thm-aff}}
\label{invariance}
In this section, we prove that for diagram $\G$ of affine type the group $W_{\G}$ is invariant under mutations.

We follow the plan of the proof from~\cite{BM}. Given two diagrams $\G_1$ and $\G_2$ related by a single mutation, we show that the groups $W_{\G_1}$ and $W_{\G_2}$ are isomorphic. For this, we investigate subdiagrams of $\G_1$ and $\G_2$ of type $\P\cup \{x\}$, where $\P$ is a subdiagram supporting some relation, 
%chordless cycle created (or destroyed) by a mutation in a vertex $k$, 
and show that the subgroup $W_{\P\cup \{x\}}$ does not change after mutations. The isomorphism can be constructed explicitly.  

Let $\G$ be a diagram of affine type, and let $W_{\G}$ be the corresponding group. For $x\in\G$ we consider $\G'=\mu_x(\G)$ and the corresponding group $W_{\G'}$. Following~\cite{BM}, we want to show that the elements $s_i'$, where   
$$
s_i'=\begin{cases}
s_xs_is_x & \text{if there is an arrow $i\to x$ in $\G$,} \\
s_i & \text{otherwise},
\end{cases}\eqno (*)
$$
satisfy the same relations as the generators of the group $W_{\G'}$. 
Since $\{s_i'\}_{i=1,\dots,n+1}$ generate  $W_{\G}$, this will mean that  the groups $W_{\G}$ and $W_{\G'}$ are isomorphic. 

\begin{remark}
\label{in-or-out}
In the definition of $s_i'$ we could freely choose to conjugate $s_i$ for outgoing arrows   $i\gets x$
rather than for incoming arrows  $i\to x$: this alteration does not affect the group with generators $\{s_i'\}$ since it is equivalent to conjugation of all generators by $s_x$.

\end{remark}

%\subsection{Invariance of the group $W_{\G}$ under mutations}
%\label{invariance}

\subsection{Pseudo-cycles and risk diagrams.}
In this section we collect elementary properties of the group $W_{\G}$ and introduce the subdiagrams we will use in the sequel.

\begin{defin}[Pseudo-cycle]
We call a subdiagram $\P$ of $\G$ a {\it pseudo-cycle} if the vertices of $\P$ form the support of some relation from the presentation of $W_{\G}$. In particular, every oriented cycle $\C\subset\G$ is a pseudo-cycle. 
%For every subdiagram $\R\subset\G$ we denote by   

\end{defin}

The following statement observed in~\cite{BM} applies without any changes in our settings.

\begin{lemma}
\label{one way}
If a mutation $\mu_x$ for $x\in \G$ preserves the group $W_{\G}$ then
the mutation $\mu_x$ preserves the group $W_{\mu_x(\G)}$  defined by the diagram $\mu_x(\G)$. 

\end{lemma}

\begin{proof}
The lemma follows from Remark~\ref{in-or-out}: performing the mutation $\mu_x:\G\to \G'$ we conjugate the ends of the 
incoming arrows, while performing the mutation  $\mu_x': \G'\to \G$ we conjugate the ends of outgoing arrows.
Then the relations we need to check for $\mu_x'$ coincide with ones we need to check for $\mu_x$.

%the fact that the defining relations of $W_\G$ do not change under reversing of all arrows in $\G$ (Remark~\ref{reversion}) and

\end{proof}

\begin{lemma}
\label{inv}
If a mutation class of some diagram of affine type contains a diagram $\G$ then it contains also the diagram $\G^{op}$ obtained from $\G$ by reversing of all arrows.

\end{lemma}

\begin{proof}
First, we mutate $\G$ to the acyclic form (or to a cycle with only two changes of the directions in the case of $\widetilde A_{p,q}$).
For the acyclic representatives (and for the cycle with two changes of orientation) one can reverse all arrows by sink/source mutations. Last, one mutates back to $\G^{op}$.

\end{proof}

\begin{lemma}
\label{in and out}
Assume that for every pseudo-cycle $\P$ the following two assumptions hold:
\begin{itemize}
\item[{\rm{(C1)}}] 
for every $x\in \P$  the group $W_{\P}$ is isomorphic to $W_{\mu_x(\P)}$ via the transformation~$(*)$;   
\item[{\rm{(C2)}}] 
for every connected diagram $\R=x\cup \P$ such that $\R\subset \G_1$ for some diagram $\G_1$ in the mutation class of $\G$ 
the group  $W_{\R}$ is isomorphic to $W_{\mu_x(\R)}$ via the transformation~$(*)$. 
\end{itemize}

Then the group  $W_{\G}$ is invariant under all mutations.

\end{lemma}

\begin{proof}
To prove the lemma it is sufficient to show that  $W_{\G}$ is preserved by each mutation.
Let $\G_1$ be a diagram mutation-equivalent to $\G$. Chose $x\in \G_1$ and consider $\mu_x$.  
We need to check that all relations of  $W_{\G_1}$ do follow from the relations of  $W_{\mu_x(\G_1)}$ and vice versa.
Since $\G_1$ and   $\mu_x(\G_1)$ play symmetric roles, Lemma~\ref{one way} implies that it is sufficient to show that for each 
pseudo-cycle $\P\subset \G_1$ the corresponding relation $r_\P$ follows from the relations of  $W_{\mu_x(\G_1)}$.

If $x\in \P$ then $r_\P$ follows from the relations of $W_{\mu_x(\P)}$ in view of assumption {\rm{(C1)}}.
If $x\notin \P$ and $x$ is connected to $\P$  then $r_\P$ follows from the relations of $W_{\mu_x(P\cup x)}$ in view of assumption {\rm{(C2)}}.
If $x\notin \P$ and $x$ is not connected to $\P$ then $r_\P$ is a relation of  $W_{\mu_x(\G_1)}$ with the same supporting diagram $\P$.
This proves the lemma.

\end{proof}

We will use the following refinement of Lemma~\ref{in and out}.

\begin{lemma}
\label{outside}
It is sufficient to check assumption {\rm{(C2)}} of Lemma~\ref{in and out} for connected diagrams $x\cup \P$ such that
there is at least one incoming arrow  to $x$ and at least one outgoing  arrow from $x$.
%\item[(b)] the number of incoming arrows at $x$ does not exceed the number of outgoing arrows.  

\end{lemma}

\begin{proof}
Suppose that $x$ is incident  in $\R=x\cup \P$ to outgoing arrows only (i.e., $x$ is a {\it source} of $\R$). Then $\mu_x$ does not change neither the subdiagram $\P$ nor the generators $s_i$ corresponding to $\P$, so $\P$ determines the same relation for both groups. Further, $x$ is not contained in any oriented cycle in $\R$. Therefore, no pseudo-cycle of order at least $3$ in $\R$ contains $x$, so no relation is changed after mutation $\mu_x$. Thus, $W_{\R}$ is isomorphic to  $W_{\mu_x(\R)}$.

If $x$ is incident  in $\R=x\cup \P$ to incoming arrows only (i.e., $x$ is a {\it sink} of $\R$), then we apply first $\mu_x$ and then use Lemma~\ref{one way} together with the result of the paragraph above. 

%Suppose that there are more incoming to $x$ arrows than outgoing from $x$ ones. First we apply Lemma~\ref{invert} to change the orientations of all arrows in $\R=x\cup \P$, denote by $\R_1$ the obtained diagram. Then we apply Lemma~\ref{inv} to see that $\R_1$ is also a subdiagram of some diagram in the mutation class of $\G$. Since $\R_1$ satisfies conditions (a) and (b) it will be checked. So, condition (2) will hold automatically for $\R=x\cup \P$.  

\end{proof}

\begin{defin}[Risk diagram]
Let $\P$ be a pseudo-cycle and let $\R=x\cup \P$ satisfy the condition of Lemma~\ref{outside}, i.e. $x$ is neither sink nor source of $\R$. Then we call $\R$ a {\it risk diagram}. We call $\R$ a {\it risk diagram for} $\G$ if $\R$ is a risk diagram and  $\R$ is a subdiagram of some diagram mutation-equivalent to $\G$. 

\end{defin}

Now we can reformulate our task using the definitions above. According to Lemma~\ref{in and out}, to show invariance of $W_{\G}$ under mutations we need to verify whether {\rm{(C1)}} holds for every pseudo-cycle and whether {\rm{(C2)}} holds for every risk diagram for $\G$. We will refer to assumptions {\rm{(C1)}} and {\rm{(C2)}} from Lemma~\ref{in and out} as to {\it Condition} {\rm{(C1)}} and {\it Condition} {\rm{(C2)}}, or simply {\rm{(C1)}} and {\rm{(C2)}}.

The following lemma is evident.

\begin{lemma}
\label{order}
Suppose that $\R$ is a risk diagram for $\G_1$, and  $\G_1\subset\G_2$. If condition {\rm{(C2)}} holds for $\R$ as a risk diagram for $\G_1$, then {\rm{(C2)}}  holds for $\R$ as a risk diagram for $\G_2$.                 
\end{lemma}

\subsection{Checking condition {\rm{(C1)}}}

Condition {\rm{(C1)}} for pseudo-cycles of size 1 or 2 (i.e. corresponding to the relations of types (R1) and (R2)) is evident.
We will first check pseudo-cycles corresponding to the relations of type (R4) (see Lemma~\ref{condition1}) and then consider ones corresponding to the relations of type (R3) (Lemma~\ref{rel-cycles}).

\begin{lemma}
\label{condition1}
Condition {\rm{(C1)}} holds for all five pseudo-cycles corresponding to additional affine relations.

\end{lemma}

The proof of the  lemma is a straightforward computation, we illustrate it by the following example.

\begin{example}
\label{ex-gt2}
Let us show that the mutation $\mu_2$ preserves the group $W_\G$ for the diagram $\G$ shown in Fig.~\ref{ex-g}. 

First, we write down the groups:
\begin{multline*}
W_\G= \langle s_1,s_2,s_3\ |\ e=s_i^2=(s_1s_2)^6=(s_2s_3)^6=\\ =(s_1s_2s_3s_2)^3=(s_2s_3s_1s_3)^6=(s_3s_1s_2s_1)^6=
(s_2s_1s_2s_1s_2s_3)^2 \rangle
\end{multline*}
and 
$$
W_{\mu_2(\G)}=\langle t_1,t_2,t_3 \ |\ e=t_i^2=(t_1t_3)^3=(t_1t_2)^6=(t_2t_3)^6=(t_2t_3t_1t_3)^2=(t_3t_1t_2t_1)^2   \rangle 
$$
where 
%\begin{equation}
%\label{s=t}
$$
t_1=s_1, \ t_2=s_2 \text{\quad   and \quad } t_3=s_2s_3s_2.\eqno (*) 
$$
%\end{equation}

We need to show that all relations of $W_{\mu_2(\G)}$ follow from the relations of $W_\G$
and equalities~$(*)$, as well as all relations of $W_{\G}$ follow from the relations of $W_{\mu(\G)}$
and equalities~$(*)$.
Let us check first that $(t_3t_1t_2t_1)^2=e $:

$$(t_3t_1t_2t_1)^2  \stackrel{(*)}{=} (s_2s_3s_2s_1s_2s_1)^2=s_2(s_3s_2s_1s_2s_1s_2)^2s_2
 \stackrel{(s_2s_1s_2s_1s_2s_3)^2=e}{=} e.
$$
Similarly,
\begin{multline*} (s_3s_1s_2s_1)^6   \stackrel{(*)}{=}  (t_2t_3t_2t_1t_2t_1)^6   \stackrel{(t_1t_2)^6=e}{=} 
t_2 (t_3 t_1t_2t_1t_2t_1t_2t_1)^6 t_2   \stackrel{(t_3t_1t_2t_1)^2=e}{=}\\ =t_2 ( t_1t_2t_1t_3t_2t_1t_2t_1)^6 t_2=
t_2t_1t_2t_1(t_3t_2)^6t_1t_2t_1t_2   \stackrel{(t_3t_2)^6=e}{=} e
\end{multline*}

All the other relations are checked similarly or even easier.

\end{example}

\begin{figure}[!h]
\begin{center}
\psfrag{G}{\small $\G$}
\psfrag{muG}{\small $\mu_2(\G)$}
\psfrag{mu2}{\scriptsize $\mu_2$}
\psfrag{1}{\scriptsize $1$}
\psfrag{2}{\scriptsize $2$}
\psfrag{3}{\scriptsize $3$}
\psfrag{3_}{\scriptsize $3$}
\psfrag{4}{\scriptsize $4$}
\epsfig{file=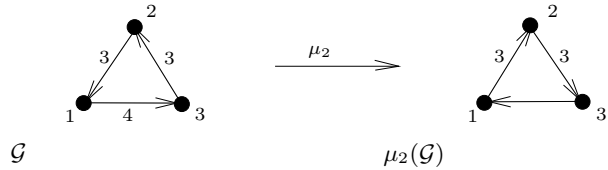,width=0.5\linewidth}
\caption{Notation for Example~\ref{ex-gt2}}
\label{ex-g}
\end{center}
\end{figure}

\begin{lemma}
\label{rel-cycles}
Condition {\rm{(C1)}} holds for a pseudo-cycle forming an oriented cycle.

\end{lemma}

\begin{proof}
In view of Lemma~\ref{cycles} a mutation-finite oriented cycle is either a simply-laced oriented cycle (finite type $D_n$)
or one of the cycles shown in Table~\ref{short cycles}. In the former case the statement follows from~\cite{BM}, in the latter case we check {\rm{(C1)}} straightforwardly (applying computation similar to one in Example~\ref{ex-gt2}). 

\end{proof}

Summarizing Lemmas~\ref{condition1} and~\ref{rel-cycles} we obtain the following corollary.
\begin{cor}
\label{all-ps}
Condition {\rm{(C1)}} holds for any pseudo-cycle in a diagram of affine type.

\end{cor}

\subsection{Condition {\rm{(C2)}} for small risk diagrams }

There are no risk diagrams $\R=x\cup \P$ with pseudo-cycles $\P$ of order $1$. In this section we check {\rm{(C2)}} for all risk diagrams with pseudo-cycles of order $2$.

\begin{lemma}
Let $\P$ be a pseudo-cycle of order $2$, and let $\R=x\cup \P$ be a risk diagram for an affine diagram $\G$. Then {\rm{(C2)}} holds for $\R$.

\end{lemma}

\begin{proof}
%The proof lemma of the lemma is straightforward (since the number of mutationally finite diagrams of order 3 is small).
%
First, suppose that $\R$ is an oriented cycle. Then $\R$ itself is a pseudo-cycle, and condition {\rm{(C2)}} for $\R$ becomes condition {\rm{(C1)}} for pseudo-cycles checked in Lemma~\ref{rel-cycles}.

Now assume that $\R$ is not an oriented cycle. Since $\R$ is a subdiagram of a diagram of affine type and contains three vertices only, it is easy to see that $\R$ is either a diagram of finite type, or a simply-laced non-oriented cycle, or a diagram of type $\widetilde C_2$ or $\widetilde G_2$.
In the former case we use results  of~\cite{BM}, in all the other cases we perform the mutation $\mu_x$ 
($x$ can be assumed to be the only vertex of $\R$ incident to both incoming and outgoing arrows)
and get an oriented cycle. Applying Lemma~\ref{one way} we obtain the statement of the lemma. 

\end{proof}

\subsection{Condition {\rm{(C2)}} for $\widetilde A_n$}
\label{sec a}

\begin{lemma}
\label{A-}
Condition {\rm{(C2)}} holds for the risk diagram $\R=x\cup \P$ shown in Fig.~\ref{l4}.

\end{lemma}

\begin{figure}[!h]
\begin{center}
\psfrag{x}{\scriptsize $x$}
\epsfig{file=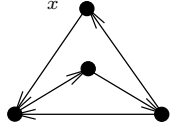,width=0.135\linewidth}
\caption{The risk diagram in Lemma~\ref{A-}}
\label{l4}
\end{center}
\end{figure}

The proof is straightforward.

\begin{lemma}
\label{lA}
Condition {\rm{(C2)}} holds for all risk diagrams $\R=x\cup\P\subset \G$ where $\G$ is mutation-equivalent to $\widetilde A_{p,q}$.
\end{lemma}

\begin{proof}
 Recall that  $\widetilde A_{p,q}$ is a block-decomposable skew-symmetric diagram 
corresponding to a triangulation of an annulus with $p$ and $q$ marked points on the boundary components.

First, suppose that $\G$ has a double arrow (recall, it is an arrow labeled by $4$).
A double arrow in a subdiagram of a diagram of type $\widetilde A_{p,q}$ can arise only from two triangles glued as in Fig.~\ref{annulus} (cf. Table~\ref{double edges}).

\begin{figure}[!h]
\begin{center}
\psfrag{a}{{\scriptsize $\alpha$}}
\psfrag{1}{}
\psfrag{2}{}
\epsfig{file=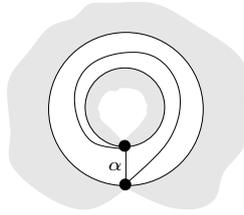,width=0.2\linewidth}
\caption{Triangulated annulus }
\label{annulus}
\end{center}
\end{figure}

Cutting the triangulation along an arc $\alpha$ corresponding to one of the ends of a double arrow we get a disk. Since the diagram of the new surface is a subdiagram of $\G$ obtained by removing the vertex corresponding to $\alpha$, this implies that 
\begin{itemize}
\item $\G$ contains at most one double arrow;
\item $\G$ looks like one of two diagrams shown in Figure~\ref{A}.
\end{itemize}

\begin{figure}[!h]
\begin{center}
\psfrag{Am}{\small $A_m$}
\psfrag{Ak}{\small $A_k$}
\epsfig{file=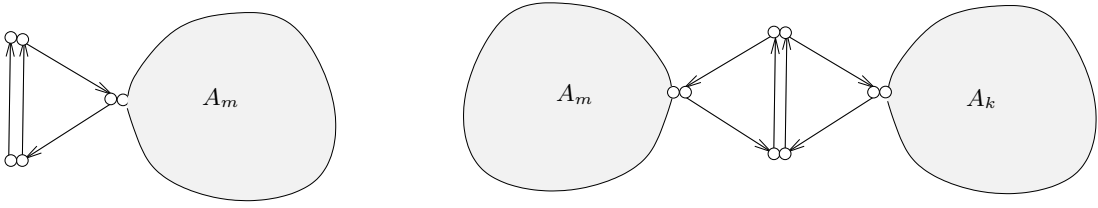,width=0.9\linewidth}
\caption{Block decomposition of diagrams of type $\widetilde A_{p,q}$ containing a double arrow. }
\label{A}
\end{center}
\end{figure}

In particular, since $\P$ is either a cycle or a pseudo-cycle of type $\widetilde A_{2,2}$ (as no other pseudo-cycle from Table~\ref{add-t} is a subdiagram of a diagram of type $\widetilde A_{p,q}$) and $x$ is connected to $\P$ by at least two arrows, every risk diagram is contained in a subdiagram of type $A_l$.
So, by  Lemma~\ref{order} and results of~\cite{BM} we see that {\rm{(C2)}} holds for all risk diagrams for $\G$.

Now, suppose that $\G$ contains simple arrows only.
In this case any pseudo-cycle is an oriented cycle $\C$. If $|\C|>4$  then the triangulated surface corresponding to $\C$ has a puncture (since the only decomposition of $\C$ in this case consists of blocks of type I), which is impossible for $\C$ being a subdiagram of a diagram of type $\widetilde A_{p,q}$.
If $|\C|\le 4$ then the decomposability of $x\cup \P$ implies that $x\cup \P$ is one of the diagrams shown in Fig.~\ref{a-small}; three of them can not be a subdiagram of a diagram of the type $\widetilde A_{p,q}$ since the corresponding surfaces have punctures, and the fourth is treated in Lemma~\ref{A-}.

\end{proof}

\begin{figure}[!h]
\begin{center}
\psfrag{x}{\scriptsize $x$}
\epsfig{file=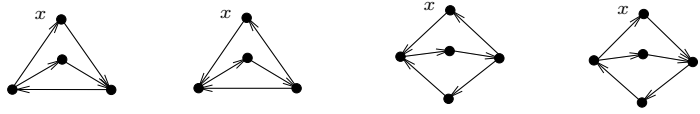,width=0.57\linewidth}
\caption{Small diagrams for the proof of Lemma~\ref{lA}: the diagram on the left is mutation-equivalent to $D_4$, the two diagrams on the right are mutation-equivalent to $\widetilde D_4$; the remaining one is checked in Lemma~\ref{A-}.}
\label{a-small}
\end{center}
\end{figure}

\subsection{Condition {\rm{(C2)}} for $\widetilde D_n$}
\label{sec d}
All the risk diagrams in this section are subdiagrams of a diagram $\G$ of type $\widetilde D_n$.

\begin{lemma}
\label{D-}
Condition {\rm{(C2)}} holds for three risk diagrams $\R=x\cup \P$ shown in Fig.~\ref{l6}.

\end{lemma}

\begin{figure}[!h]
\begin{center}
\psfrag{4}{\scriptsize $4$}
\psfrag{x}{\scriptsize $x$}
\psfrag{u}{\scriptsize $u$}
\psfrag{v}{\scriptsize $v$}
\psfrag{a}{\scriptsize $a$}
\psfrag{b}{\scriptsize $b$}
\epsfig{file=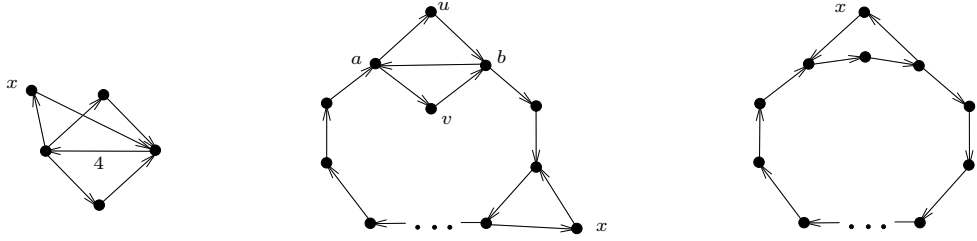,width=0.8\linewidth} %\qquad 
\caption{The diagrams in Lemma~\ref{D-}. The diagram in the middle has any size $|S|\ge 6$, the node $x$ is attached to the ends of any arrow not lying in the subdiagram $uvab$. The diagram on the right has any size $|S|\ge 5$.}
\label{l6}
\end{center}
\end{figure}

The proof is straightforward.

\begin{lemma}
\label{lD}
Condition {\rm{(C2)}} holds for all risk diagrams $\R=x\cup\P\subset \G$ where $\G$ is mutation-equivalent to $\widetilde D_n$.
\end{lemma}

\begin{proof}
Recall that diagrams of type  $\widetilde D_n$ correspond to ideal triangulations of a twice punctured disk (with $n-2$ marked points on the boundary).

\begin{figure}[!h]
\begin{center}
\psfrag{a}{\small (a)}
\psfrag{b}{\small (b)}
\psfrag{c}{\small (c)}
\psfrag{d}{\small (d)}
%\psfrag{Am}{\small $X$}
\psfrag{Ak}{\small $X$}
%\psfrag{1}{\scriptsize $p_1$}
%\psfrag{2}{\scriptsize $p_2$}
\epsfig{file=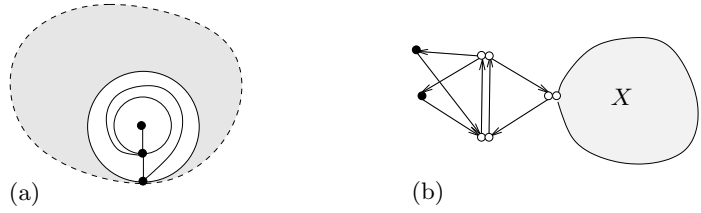,width=0.57\linewidth}
\caption{Triangulation of a twice punctured disk and its diagram. 
$X$ is a diagram of type $A_{n-3}$.}
\label{d-triang}
\end{center}
\end{figure}

First, suppose that $\G$ contains a double arrow. As it is shown in Table~\ref{double edges}, a double arrow in a diagram of the type $\widetilde D_n$ can be obtained by gluing blocks of type II and IV only.
The gluing of these two blocks results in a disk with two punctures and one marked point on the boundary, as shown in Fig.~\ref{d-triang}(a) (denote this disk by $\mathcal A$ and the whole twice punctured disk corresponding to the whole diagram $\G$ by $\mathcal D$).
Clearly, $\mathcal G\setminus \mathcal A$ is a disk (corresponding to a diagram of type $A_{n-3}$, so the diagram $\G$ is constructed as in Fig.~\ref{d-triang}(b). 
In particular, this means that each risk diagram is either contained in a subdiagram of type $A_k$, or is one shown in Fig.~\ref{l6} on the left. Thus, each risk subdiagram of $\G$ is already checked either in~\cite{BM} or in Lemma~\ref{D-}.  

\begin{figure}[!h]   
\begin{center}
\psfrag{u}{\scriptsize $u$}
\psfrag{v}{\scriptsize $v$}
\psfrag{A}{\small $A$}
\psfrag{a}{\small (a)}
\psfrag{b}{\small (b)}
\psfrag{c}{\small (c)}
\psfrag{d}{\small (d)}
\psfrag{e}{\small (e)}
\psfrag{f}{\small (f)}
\psfrag{p}{\scriptsize $p$}
\epsfig{file=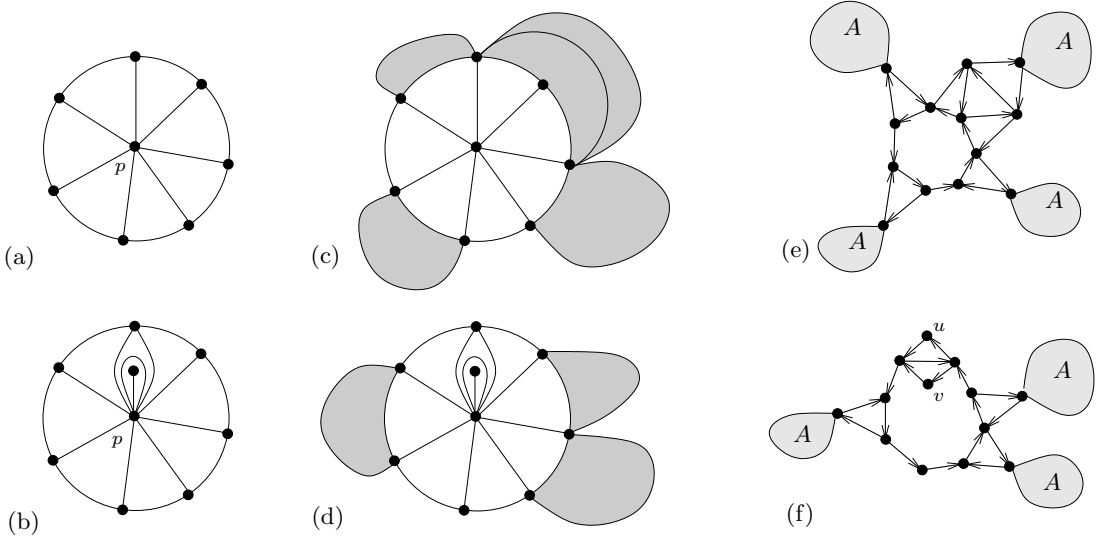,width=0.9\linewidth}
\caption{(a),(b): possible neighborhoods of a puncture (up to flip in an interior edge of the digon in (b)); (c),(d): the same with disks attached (also up to a flip in (d)); (e),(f): the corresponding diagrams (up to mutation in the vertices $v$ and $u$); diagrams marked by $A$ all have types $A_{k_i}$.}
\label{d-u}
\end{center}
\end{figure}

Now, suppose that $\G$ contains simple arrows only. Consider a puncture $p$ inside the twice punctured disk, let $\mathcal U$ be the union of all triangles incident to $p$.
Then $\mathcal U$ is triangulated in one of the two ways shown in Fig.~\ref{d-u}(a) and (b). The remaining part  $\mathcal D\setminus \mathcal U$ of the twice punctured disk  $\mathcal D$ is attached to $\mathcal U$ 
in such a way that either only one new puncture arises (for the diagram on  Fig.~\ref{d-u}(a)) or no new puncture arises (for the diagram of  Fig.~\ref{d-u}(b)). This is possible only if we attach some disks (or nothing at all) to some boundary edges of $\mathcal U$,  which results in the triangulations looking as in
 Fig.~\ref{d-u}(c) and  (d) respectively, and corresponds to diagrams on  Fig.~\ref{d-u}(e) and (f).
It is easy to see from these diagrams that all risk subdiagrams of $\G$ are already checked either in~\cite{BM} or in Section~\ref{sec a} or in Lemma~\ref{D-}.
  
\end{proof}

\subsection{Condition {\rm{(C2)}} for $\widetilde B_n$ and  $\widetilde C_n$ }
\label{sec bc}

\begin{lemma}
\label{B-}
Condition {\rm{(C2)}} holds for the risk diagram $\R=x\cup \P$ shown in Fig.~\ref{l9}.

\end{lemma}

\begin{figure}[!h]
\begin{center}
\psfrag{2}{\scriptsize $2$}
\psfrag{4}{\scriptsize $4$}
\psfrag{x}{\scriptsize $x$}
\psfrag{u}{\scriptsize $u$}
\psfrag{v}{\scriptsize $v$}
\psfrag{a}{\scriptsize $a$}
\psfrag{b}{\scriptsize $b$}
\epsfig{file=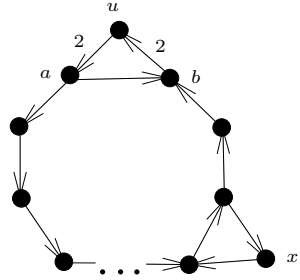,width=0.24\linewidth}
\caption{The diagram in Lemma~\ref{B-}. The diagram has any size $|\R|\ge 5$. The node $x$ is attached to the ends of any arrow not lying in the subdiagram $uab$.}
\label{l9}
\end{center}
\end{figure}

The proof is straightforward.

\begin{lemma}
\label{lBC}
Condition {\rm{(C2)}} holds for all risk diagrams $\R=x\cup\P\subset \G$ where $\G$ is mutation-equivalent to $\widetilde B_n$ or $\widetilde C_n$.

\end{lemma}

\begin{proof}
The diagrams of type  $\widetilde B_n$ and $\widetilde C_n$ correspond to ideal triangulations of a punctured disk with one orbifold point and a disk with two orbifold points respectively, 
see Table~\ref{surface realizations}.

%\begin{figure}[!h]
%\begin{center}
%\psfrag{Am}{\small $X$}
%\psfrag{Ak}{\small $Y$}
%\psfrag{1}{\scriptsize $p1$}
%\psfrag{2}{\scriptsize $p2$}
%\epsfig{file=./pic/bc-triang.eps,width=0.7\linewidth}
%\caption{Triangulations corresponding to the diagrams $\widetilde B_n$ and $\widetilde C_n$.}
%\label{bc-triang}
%\end{center}
%\end{figure}

First, suppose that $\G$ contains a double arrow. 
A double arrow is either contained in a block $\widetilde{\rm{V}}_{12}$ or is obtained by gluing two blocks of types  II and $\widetilde{\rm{IV}}$, see Table~\ref{double edges}.  
In the former case the triangulation looks as in Fig.~\ref{bc-diagr}(a)  and the diagram $\G$ looks as  in Fig.~\ref{bc-diagr}(b), so $\G$ does not contain any risk diagrams that were not studied yet.
In the latter case we obtain the triangulation shown in  Fig.~\ref{bc-diagr}(c) which results in a diagram shown in  Fig.~\ref{bc-diagr}(d).
Hence, each risk subdiagram of $\G$ is already checked either in~\cite{BM} or in Sections~\ref{sec a} and~\ref{sec d} or in Lemma~\ref{l9}.

\begin{figure}[!h]
\begin{center}
\psfrag{a}{\small (a)}
\psfrag{b}{\small (b)}
\psfrag{c}{\small (c)}
\psfrag{d}{\small (d)}
\psfrag{e}{\small (e)}
\psfrag{f}{\small (f)}
\psfrag{2}{\small 2}
\psfrag{4}{\small 4}
\psfrag{X}{\small $X$}
\psfrag{Y}{\small $Y$}
\psfrag{Am}{\small $A_m$}
\psfrag{Bk}{\small $B_k$}
\epsfig{file=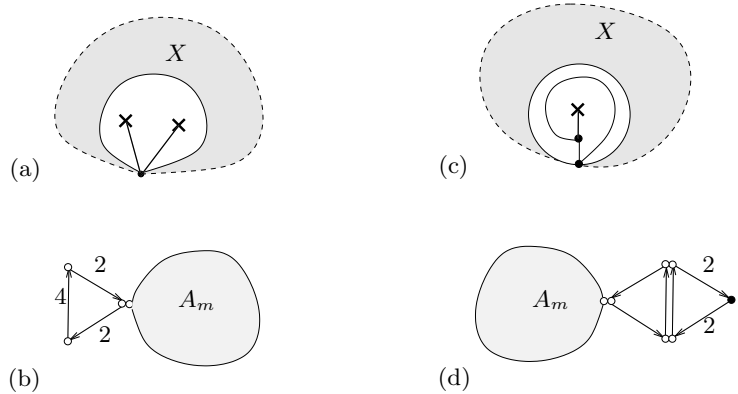,width=0.6\linewidth}
\caption{Diagrams of type  $\widetilde B_n$ and  $\widetilde C_n$ with double arrows.
$X$ is a disk. }
\label{bc-diagr}
\end{center}
\end{figure}

Suppose now that $\G$ contains no double arrows. Then any pseudo-cycle contained in $\G$ is either a simply-laced cycle or a pseudo-cycle of type $\widetilde D_k$ or  $\widetilde B_k$. 
Consider these three types of pseudo-cycles separately.

A pseudo-cycle of type $\widetilde D_k$ can not be a subdiagram of $\G$ as  the triangulated surface corresponding to $\widetilde D_k$ has two punctures, while the surface corresponding to $\G$ has 
either one puncture (if $\G$ is of type  $\widetilde B_n$)  or no punctures  (if $\G$ is of type  $\widetilde C_n$).

A pseudo-cycle of type  $\widetilde B_k$ corresponds to a triangulated disk with a puncture and an orbifold point, so, it can not be a subdiagram of $\G$ if $\G$ is of mutation type  $\widetilde C_n$.
If  $\G$ is of mutation type  $\widetilde B_n$ then the triangulation of a surface corresponding to $\G$ is obtained from a triangulation of a surface corresponding to $\P$ by attaching a number of disks
(see Fig.~\ref{bc-}(a)), and the diagram $\G$ looks as in Fig.~\ref{bc-}.b. The only new risk subdiagram in $\G$ 
is the diagram checked in Lemma~\ref{B-}.

Finally, consider a pseudo-cycle $\P$ which is a simply-laced cycle. Let $\R=x\cup \P$ be a risk diagram for $\G$. Consider a block decomposition of $\R$.
If all arrows incident to $x$ in $\R$ are simple, then $\R$ is a skew-symmetric block-decomposable diagram and is already checked  either in~\cite{BM} or in Sections~\ref{sec a} and~\ref{sec d}. 
So, we may assume that some arrow incident to $x$ is labeled by $2$. Furthermore, since $x$ is attached to $\P$ by at least 2 arrows (by the definition of risk diagram), all blocks containing $x$ have at least two white vertices.
The only block containing a non-simple arrow and two white vertices is the block $\widetilde{\rm{IV}}$. So, $x$ is attached to the simply-laced cycle $\P$ by the block  $\widetilde{\rm{IV}}$
and we get a diagram shown in Fig.~\ref{bc-}(c). After mutation in $x$ this diagram coincides with a pseudo-cycle of type $B_k$, so, {\rm{(C2)}} for $\R$ becomes {\rm{(C1)}} for pseudo-cycle of type $B_k$ which is already verified (See Cor.~\ref{all-ps}).
 
\end{proof}

\begin{figure}[!h]
\begin{center}
\psfrag{a}{\small (a)}
\psfrag{b}{\small (b)}
\psfrag{c}{\small (c)}
\psfrag{x}{\small $x$}
\psfrag{A}{\small $A$}
\psfrag{Bk}{\small $\t B_k$}
\epsfig{file=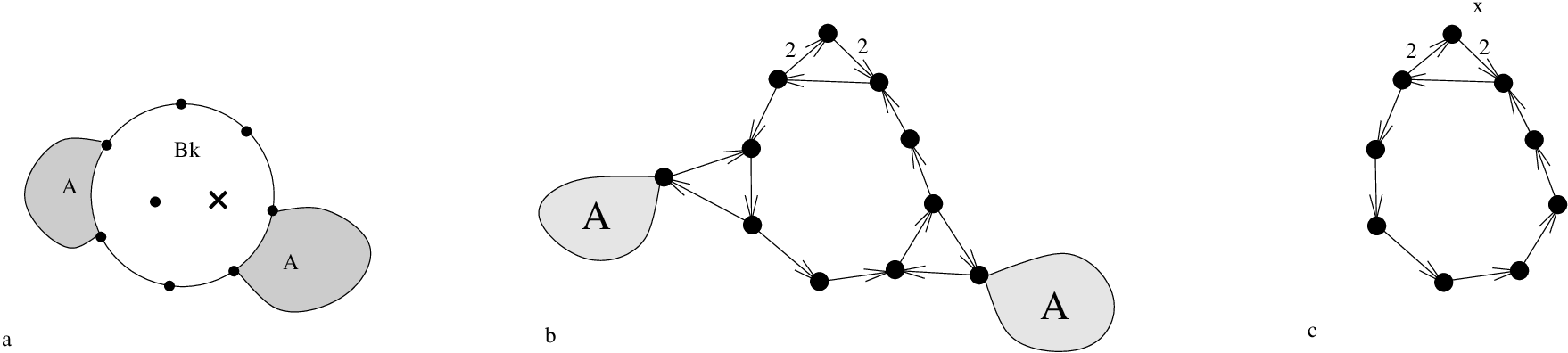,width=0.9\linewidth}
\caption{To the proof of Lemma~\ref{lBC}}
\label{bc-}
\end{center}
\end{figure}

\subsection{Condition {\rm{(C2)}} for $\widetilde G_2$}
This is a small diagram with a small mutation class, so we just check {\rm{(C2)}} explicitly.

\subsection{Condition {\rm{(C2)}} for $\widetilde F_4$}
Consider the mutation class of $\widetilde F_4$ (it consists of 59 diagrams).
We need to find  all pseudo-cycles and all risk diagrams.

First, consider proper subdiagrams.
Let $\G$ be a diagram of type  $\widetilde F_4$. Then each proper subdiagram  $X\subset \G$ has order at most $4$
and contains no arrows labeled by $3$. Due to results of~\cite{FeSTu2}, this implies that either $X$ is block-decomposable or $X$ is mutation-equivalent to $F_4$.
All risk diagrams contained in diagrams of mutation type $F_4$ satisfy {\rm{(C2)}} by results of~\cite{BM}. 
Any block-decomposable affine diagram is of one of the types $\widetilde A_{p,q}$,  $\widetilde B_n$, $\widetilde C_n$ or  $\widetilde D_n$, and thus all risk diagrams are already checked in the previous sections. 
Hence, {\rm{(C2)}} for risk subdiagrams of order at most $4$ is verified.

Looking through the mutation class, we find a unique risk diagram of order $5$, see Fig.~\ref{f-risk}(a).
We label by $x$ the vertex not lying in the pseudo-cycle of order 4
(so that $\R=\P\cup x$ where $\R$ is the risk diagram and $\P$ is a pseudo-cycle).
It is easy to see that the mutation $\mu_x$ turns $\R$ into the cyclic diagram on  Fig.~\ref{f-risk}(b).
This diagram is a pseudo-cycle checked in Lemma~\ref{rel-cycles}.

This proves that {\rm{(C2)}} holds for all risk diagrams for $\widetilde F_4$.

\begin{figure}[!h]
\begin{center}
\psfrag{a}{\small (a)}
\psfrag{b}{\small (b)}
\psfrag{c}{\small (c)}
\psfrag{x}{\small $x$}
\psfrag{2}{\small $2$}
\epsfig{file=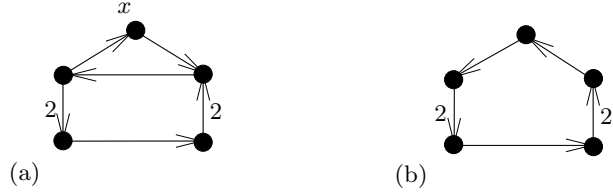,width=0.5\linewidth}
\caption{Risk diagram of order 5 for $\widetilde F_4$.}
\label{f-risk}
\end{center}
\end{figure}

%and see that each risk diagram is a subdiagram of some diagram $\widetilde D<\widetilde F_4$.
%
%
%More precisely, if a risk diagram $\R$ is not a subset of some smaller diagram $\widetilde D<\widetilde F_4$ then $|S|=5$, however, the mutation class of  $\widetilde F_4$.
%Contains no diagram of type $x\cup \P$ where $\P$ is a pseudo-cycle.

\subsection{Condition {\rm{(C2)}} for $\widetilde E_6$,  $\widetilde E_7$, $\widetilde E_8$}

Since  $\widetilde E_6$,  $\widetilde E_7$, $\widetilde E_8$ are skew-symmetric, the only types of pseudo-cycles we need to check are simply-laced oriented cycles  and pseudo-cycles
of types $\widetilde A_{p,q}$ and  $\widetilde D_k$ . First we show (Lemma~\ref{simple}) that no sufficiently large risk diagram contains double arrow, then prove (Lemmas~\ref{an},~\ref{l-wheel} and~\ref{dn})
that the risk diagrams are block-decomposable, and finally, in Lemma~\ref{lE} we show that {\rm{(C2)}} holds for risk subdiagrams of diagrams of type $\widetilde E_6$,  $\widetilde E_7$, $\widetilde E_8$.

\begin{nota}
Given an arrow with ends $u$ and $v$ we will call it $uv$ if the arrow points to $v$.

\end{nota}

\begin{lemma}
\label{simple}
Let $\P$ be an oriented cycle or a pseudo-cycle of type $\widetilde D_k$.
Suppose that $|\P|>4$ and $\R=x\cup \P$ is a subdiagram of some mutation-finite skew-symmetric diagram $S$.
%If $x$ is connected to $\P$ by at least two arrows 
Then $\R$ is simply-laced.

\end{lemma}

\begin{proof}
First, we note that neither oriented cycles nor pseudo-cycles of type $\widetilde D_k$ contain double arrows if their order is more than $4$, so we need to show that $x$ is not attached to $\P$ by double arrow.
  
Suppose that $x$ is connected to $\P$ by a double arrow $xa$, $a\in \P$ (the case when the arrow $ax$ is not simple is similar). Let $b$ be a neighbor of $a$ in $\P$ such that $\P$ contains arrow $ba$ (such a neighbor exists since no pseudo-cycle of order at least $3$ contains a source).
The subdiagram $U=\{xab\}\subset \R$ is a skew-symmetric mutation-finite diagram of order $3$ with a sink $a$. However, it is easy to check that any skew-symmetric mutation-finite diagram of order $3$ containing a double arrow is an oriented cycle. 
\end{proof}

In the following three lemmas we show that  for $\P$ either cyclic or of type $\widetilde A_{2,2}$ or $\widetilde D_k$ any risk diagram $\R=x\cup \P$ is block-decomposable. 

\begin{lemma}
\label{an}
Let $\P$ be a pseudo-cycle of the type  $\widetilde A_{2,2}$ or $\widetilde D_{4}$.  Suppose that $\R=x \cup \P$ is a mutation-finite skew-symmetric risk diagram. Then $\R$ is block-decomposable.

\end{lemma}

\begin{proof}
The diagram  $\R=x\cup \P$ is a mutation-finite skew-symmetric diagram of order 5 or 6, so to prove that it is block-decomposable one needs to show that it is not mutation-equivalent to $E_6$ or $X_6$ which is evident for $\widetilde A_{2,2}$ and can be done easily for $\widetilde D_4$.

\end{proof}

\begin{lemma}
\label{l-wheel}
Let $\P$ be an oriented simply-laced cycle. %, $|\P|\ge 5$. 
Suppose that $\R=x\cup \P$ is a simply-laced risk diagram, and $\R$ is mutation-finite. 
Then $\R$ is block-decomposable. 

\end{lemma}

\begin{proof}
Let $\P=\{a_1,\dots,a_n\}$ where $a_i$ is connected to $a_{i+1}$ by an arrow pointing to $a_{i+1}$ (with assumption $a_{n+1}=a_1$). Note that we may assume that $n\ge 5$: all skew-symmetric mutation-finite diagrams of order less than six are block-decomposable.

By the definition of risk diagram, $x$ is connected to $\P$ by both incoming and outgoing arrows. 
Without loss of generality we may assume that $\R$ contains an arrow $xa_1$, see Fig.~\ref{wheel}(a). 
If $a_2x$ is the only other arrow incident to $x$ then $\P\cup x$ is clearly decomposable (into block $a_1a_2x$ of type II and others of type I, see  Fig.~\ref{wheel}(b)),
so we assume that $\R$ contains some arrows $a_ix$ for $i>2$. Then there exists a unique cycle $\C$ containing 
$x,a_1$ and $a_n$, and this cycle is non-oriented. By Lemma~\ref{non-or}, this implies that each vertex of $\P$ is connected to $\C$ by even number of arrows. 
Let $l=\min\{i\ |\ a_i\in \C,i>1\}$. By assumption, $l>2$ , and there is one of the arrows $xa_l$ or $a_lx$.

If $l=3$ then $a_2$ is not connected to $x$ (otherwise there is an odd number of arrows connecting $a_2$ and $\C$). Thus, $a_3$ is connected to $x$ by the arrow $a_3x$ (since it is the only incoming arrow for $x$, see  Fig.~\ref{wheel}(c)), and this diagram is clearly block-decomposable (into blocks $xa_1a_2$, $xa_2a_3$ and $a_ia_{i+1}$ for  $3\le i \le n$).  

If $3<l\le n$ then there is an arrow $xa_2$ (or $a_2x$) and an arrow $xa_{l-1}$ (or $a_{l-1}x$), otherwise Lemma~\ref{non-or} does not hold for $a_2$ or $a_{l-1}$.
By the same reason, none of $a_i$ (for $2<i<l-1$) is connected to $x$, see   Fig.~\ref{wheel}(d). This diagram satisfies  Lemma~\ref{non-or} only if both triangles $xa_1a_2$ and $xa_{l-1}a_l$ are oriented, which implies that the diagram is block-decomposable (into these two blocks of type II and others of type I). 

%If $l=n$ then there is an arrow $xa_{n-1}$ (or $a_{n-1}x$) and no arrows connecting $x$ with $a_i$, $2<i<n-1$,  see   Fig.~\ref{wheel}.e.
% This diagram satisfies  Lemma~\ref{non-or} only if both triangles are oriented, which implies that the diagram is decomposable. % (here we use also $n\ge 5$).

\end{proof}

\begin{figure}[!h]
\begin{center}
\psfrag{a}{\small (a)}
\psfrag{b}{\small (b)}
\psfrag{c}{\small (c)}
\psfrag{d}{\small (d)}
\psfrag{e}{\small (e)}
\psfrag{x}{$\scriptsize x$}
\psfrag{a1}{$\scriptsize a_1$}
\psfrag{a2}{$\scriptsize a_2$}
\psfrag{a3}{$\scriptsize a_3$}
\psfrag{al}{$\scriptsize a_l$}
\psfrag{al1}{$\scriptsize a_{l-1}$}
\psfrag{an}{$\scriptsize a_n$}
\epsfig{file=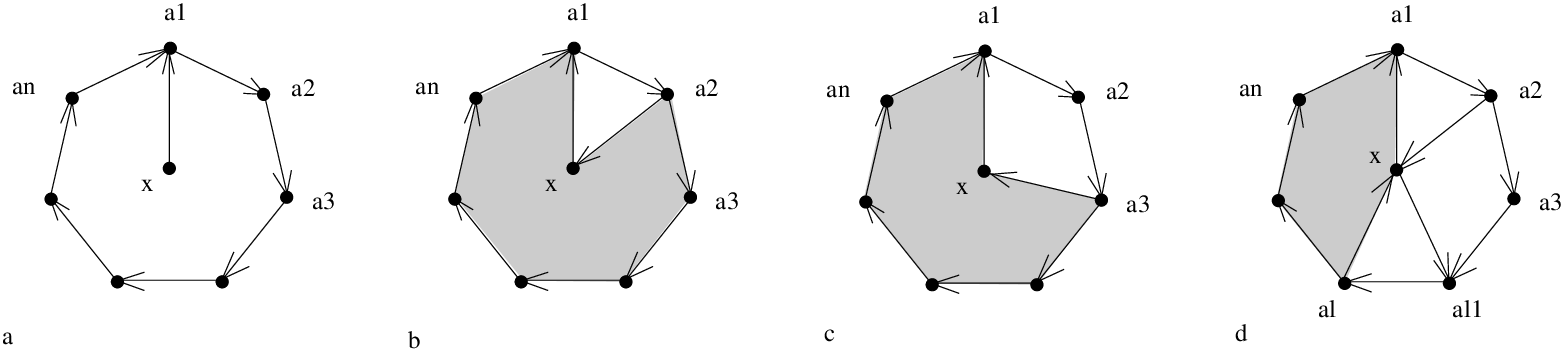,width=0.97\linewidth}
\caption{To the proof of Lemma~\ref{l-wheel}}
\label{wheel}
\end{center}
\end{figure}

\begin{remark}
\label{bl-I}
Note that the block decomposition obtained in Lemma~\ref{l-wheel} consists of one or two blocks of type II containing $x$ and several blocks of type I; in particular, if a vertex $t$ of $\C$ is not connected to $x$ then it is contained in two blocks of type I. 
\end{remark}

\begin{lemma}
\label{dn}
Let $\P$ be a pseudo-cycle of the type  $\widetilde D_k,$ $k\ge 5$.  Suppose that $\R=x\cup \P$ is a simply-laced risk diagram, and $\R$ is mutation-finite. Then $\R$ is block-decomposable.
\end{lemma}

\begin{proof}
%By Lemma~\ref{simple}, all arrows of $x\cup \P$ are simple.
The pseudo-cycle $\P$ consists of a ``big cycle'' $\C$ with one arrow $ab$ reversed and two more vertices $u$ and $v$, see Fig.~\ref{l-dn}. The cycle $\C$ is non-oriented, so it is connected to $x$ by an even number $r$ of arrows.

\begin{figure}[!h]
\begin{center}
\psfrag{x}{$\scriptsize x$}
\psfrag{a_}{$\scriptsize a$}
\psfrag{b_}{$\scriptsize b$}
\psfrag{u}{$\scriptsize u$}
\psfrag{v}{$\scriptsize v$}
\psfrag{ct}{$\scriptsize c_1$}
\psfrag{c1}{$\scriptsize c_{k-3}$}
\psfrag{al1}{$\scriptsize a_{l-1}$}
\psfrag{an}{$\scriptsize a_n$}
\epsfig{file=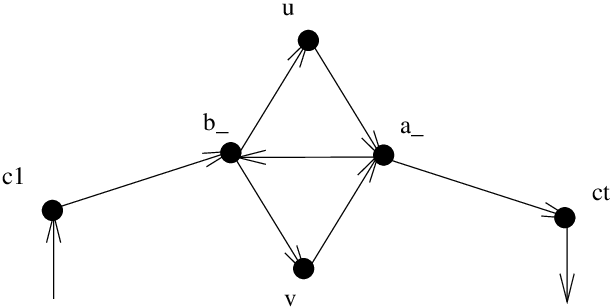,width=0.27\linewidth}
\caption{To the proof of Lemma~\ref{dn}}
\label{l-dn}
\end{center}
\end{figure}

First, suppose that $r=0$. Then both incoming to $x$ and outgoing from $x$ arrows belong to the subdiagram $xuv$. 
Then $xuav$ is a non-oriented cycle connected to $b$ by three arrows which is impossible by Lemma~\ref{non-or}.

Now, suppose that $r>0$. As the next step of the proof, we want to show that $x\cup \C$ is block-decomposable. 
%and the valence of $x$ in $\R$ is either $2$ or $4$. 
Then we will extend the block-decomposition of $x\cup \C$ to a  block-decomposition of $x\cup \P$.\\

\noindent
{\bf Claim:} The subdiagram $x\cup \C$ is block-decomposable.

\begin{proof}[Proof of Claim]
Assume first that the only vertices of $\C$ connected to $x$ are $a$ and $b$. If the subdiagram $xab$ is a non-oriented cycle, then the subdiagram $xabc_{1}c_{k-3}$ is mutation-infinite~\cite{Kel}, so we can assume that $xab$ is an oriented cycle. Then $x\cup \C$ can be decomposed into a block $xab$ of type II and others of type I. 

Now assume that there is an arrow connecting $x$ and $c_i$. Then the proof follows the proof of Lemma~\ref{l-wheel} verbatim. 

\end{proof}

To transform the decomposition of $x \cup \C$ to a decomposition of $x\cup \P$ we consider three cases: either $x$ is not connected to neither $a$ nor $b$, or $x$ is connected to exactly one of them, or it is connected to both. Our goal is to show that in all these cases the arrow $ab$ is a block of type I in a block decomposition of $x\cup\C$, and $x$ is connected to neither $u$ nor $v$: this means we can substitute $ab$ by a block $abuv$ of type IV to obtain a block decomposition of $\R$. 

\medskip
\noindent
{\bf Case 1:} $x$ is connected neither to $a$ nor to $b$.

First, we will show that $x$ is connected neither to $u$ nor to $v$. 

Suppose the contrary. 
If $x$ is connected to both $u$ and $v$ then $avxu$ is a non-oriented cycle connected to $b$ by three arrows, which is impossible by Lemma~\ref{non-or}, see Fig.~\ref{e-cl}(a). So, suppose that $x$ is connected to one of $u$ and $v$, say to $v$. Recall that $x$ is connected to $\P$ by at least two arrows. Since $x$ is not connected to $a,b,u$ we conclude that there exists $t\in \C$ such that $t$ is connected to $x$ (see Fig.~\ref{e-cl}(b)). Denote by $\C_a$ and $\C_b$ the subdiagrams in $x\cup \P$  such that $\C_a$ and $\C_b$ are chordless cycles containing $x$, $v$ and either  $a$ or $b$ respectively. Clearly, at least one of $\C_a$ and $\C_b$ is non-oriented. On the other hand, $u$ is connected to each of   $\C_a$ and $\C_b$
by a unique arrow, so we come to a contradiction. 

\begin{figure}[!h]
\begin{center}
\psfrag{c1}{}
\psfrag{a_}{\small (a)}
\psfrag{b_}{\small (b)}
\psfrag{x}{\small $x$}
\psfrag{u}{\small $u$}
\psfrag{v}{\small $v$}
\psfrag{t}{\small $t$}
\psfrag{a}{\small $a$}
\psfrag{b}{\small $b$}
\epsfig{file=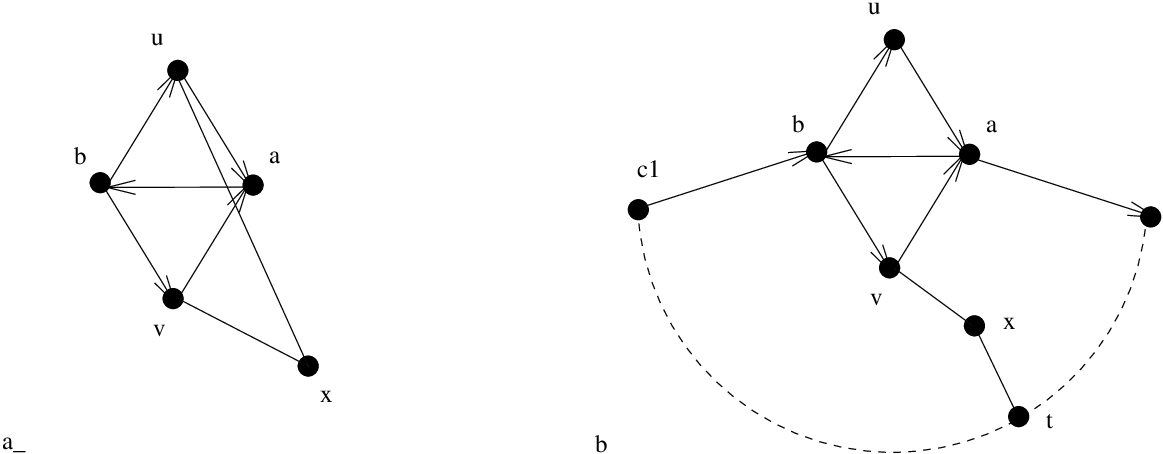,width=0.5\linewidth}
\caption{To the proof of Lemma~\ref{dn}, Case~1.}
\label{e-cl}
\end{center}
\end{figure}

Therefore, we can transform the decomposition of  $x\cup C$ into  a decomposition of $x\cup \P$ by substituting a block $ab$ of type I (see Remark~\ref{bl-I})
by a block $abuv$ of type IV.

\medskip
\noindent
{\bf Case 2:} $x$ is connected to exactly one of $a$ and $b$, say $a$ (the case when $x$ is connected to $b$ can be obtained by changing directions of all arrows). 

Suppose first that $x$ is connected in $\C\setminus a$ to $c_1$ only (see Fig~\ref{e-xabc}(a)). Then the cycle $xac_1$ is oriented (since $b$ is attached to it by one arrow only), and $ab$ is a block of type I in the block decomposition of $x\cup\C$. Let us show that $x$ is connected neither to $u$ nor to $v$, so $ab$ can be substituted by a block $abuv$ of type IV to produce a block decomposition of $\R$. Indeed, if $x$ is joined with both $u$ and $v$, then $c_{k-3}$ is connected to a non-oriented cycle $xubv$ by one arrow, which contradicts Lemma~\ref{non-or}; if $x$ is joined with one of $u$ and $v$ (say $u$), then $v$ is connected to a non-oriented cycle $xua$ by exactly one arrow, which also leads to a contradiction.   

Now assume that $x$ is connected to some $c_i,\ i>1$ in $\C\setminus a$. Let $xabc_{k-3}\dots c_m$ be the smallest chordless cycle in $x\cup\C$ containing $xab$ (it clearly does exist in this case). Note that the cycle $xabc_{k-3}\dots c_{m}$ is non-oriented (see Fig~\ref{e-xabc}(b)), so each of $u$ and $v$ is connected to it by even number of arrows. This implies that $x$ is not connected neither to $u$ nor 
to $v$. Furthermore, since $b$ is not connected to $x$, the arrow $ab$ in the decomposition of $x\cup \C$  is represented by a block of type I. Substituting this block by a block of type IV we obtain a  block decomposition of $\R$.

\begin{figure}[!h]
\begin{center}
\psfrag{a}{\scriptsize $a$}
\psfrag{b}{\scriptsize $b$}
\psfrag{a_}{{\small \!\!\!(a)}}
\psfrag{b_}{{\small (b)}}
\psfrag{u}{\scriptsize $u$}
\psfrag{v}{\scriptsize $v$}
\psfrag{c1}{\scriptsize $c_{k-3}$}
\psfrag{ct}{\scriptsize $c_1$}
\psfrag{cm}{\scriptsize $c_m$}
\psfrag{x}{\scriptsize $x$}
\psfrag{A}{\small $A_m$}
\psfrag{Bk}{\small $B_k$}
\epsfig{file=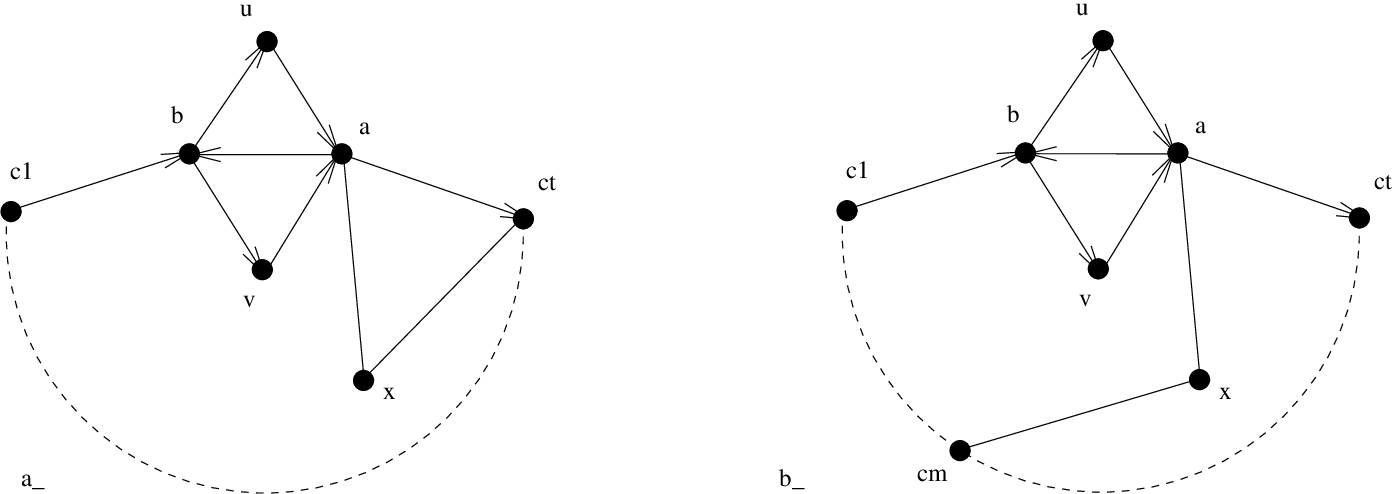,width=0.6\linewidth}
\caption{To the proof of Lemma~\ref{dn}, Case 2. }
\label{e-xabc}
\end{center}
\end{figure}

\medskip
\noindent
{\bf Case 3:} $x$ is connected to both $a$ and $b$. 

An application of Lemma~\ref{non-or} to any simply-laced diagram whose underlying graph is the complete graph on four vertices shows that such a diagram is mutation-infinite. Therefore, considering the subdiagrams $xabu$ and $xabv$ we conclude that $x$ is connected neither to $u$ nor to $v$. Thus, the subdiagram $xabuv$ looks as shown in Fig.~\ref{e-final}(a).

Since the diagram shown in Fig.~\ref{e-final}(b) is mutation-infinite for all directions of arrows incident to $x$, we conclude that $x$ is connected to $c_1$ and $c_{k-3}$, see  Fig.~\ref{e-final}(c). Furthermore, the cycle $bxc_{k-3}$ is oriented, since $u$ is connected to $bxc_{k-3}$ by a unique arrow.
Similarly, the cycle $axc_1$ is oriented, which defines the directions of all arrows in $x\cup \P$. Note that $x$ is not connected to other $c_i\in\C$: in that case either $c_1$ or $c_{k-3}$ would be connected to a non-oriented cycle by a unique arrow in contradiction with Lemma~\ref{non-or}. Now a block decomposition of $\R$ can be obtained in the same way as in the previous cases, see Fig.~\ref{e-final}(d). 

\end{proof}

\begin{figure}[!h]
\begin{center}
\psfrag{x}{\scriptsize $x$}
\psfrag{a_}{\scriptsize $a$}
\psfrag{b_}{\scriptsize $b$}
\psfrag{u}{\scriptsize $u$}
\psfrag{v}{\scriptsize $v$}
\psfrag{c1}{\scriptsize $c_{k-3}$}
\psfrag{ct}{\scriptsize $c_1$}
\psfrag{a}{\small (a)}
\psfrag{b}{\small (b)}
\psfrag{c}{\small (c)}
\psfrag{d}{\small (d)}
\psfrag{e}{\small (e)}
\psfrag{f}{\small (f)}
\epsfig{file=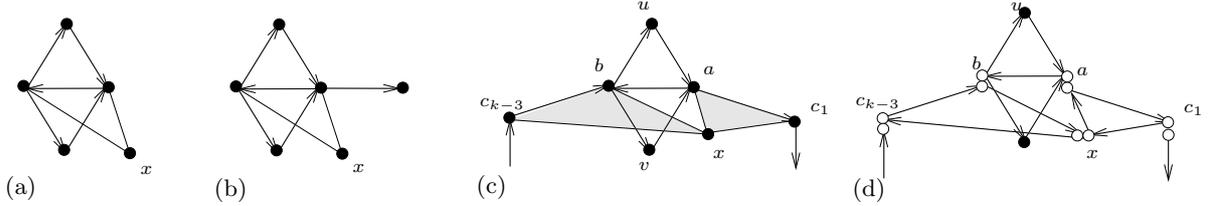,width=0.98\linewidth}
\caption{To the proof of Lemma~\ref{dn}, Case 3. }
\label{e-final}
\end{center}
\end{figure}

We summarize the results of Lemmas~\ref{simple}--\ref{dn} in the following corollary.

\begin{cor}
\label{sum}
Let $\P$ be a simply-laced oriented cycle or a pseudo-cycle of type $\widetilde A_{2,2}$ or $\widetilde D_k$. 
Let $\R=\P\cup x$ be a mutation-finite skew-symmetric risk diagram. Then $\R$ is block-decomposable.

\end{cor}

\begin{lemma}
\label{lE}
Condition {\rm{(C2)}} holds for all risk diagrams of types $\widetilde E_n$, $n=6,7,8$.

\end{lemma}

\begin{proof}
By Lemma~\ref{subd of aff}, any risk subdiagram of $\widetilde E_n$
is either of finite or of affine type.  By Corollary~\ref{sum}, all these risk subdiagrams are block-decomposable. 
So, any risk diagram  is a block-decomposable skew-symmetric diagram of
finite or affine type, i.e. any risk diagram is of mutation type $A_k$, $D_k$, $\widetilde A_{p,q}$ or $\widetilde D_k$ and is already checked 
in~\cite{BM} or in Sections~\ref{sec a} and~\ref{sec d}.  
Therefore, {\rm{(C2)}} holds for these risk diagrams. 

\end{proof}

We verified conditions {\rm{(C1)}} and {\rm{(C2)}} for all pseudo-cycles and risk diagrams for all affine diagrams. 
By Lemma~\ref{in and out}, this completes the proof of Theorem~\ref{thm-aff}.

\section{Examples of non-isomorphic groups $W$ and $\t W_{\G}$}
\label{difference}

In this section, we show that for every affine Weyl group $W$ except $\t C_n$ and $\t A_{p,1}$ (cf. Remark~\ref{cn}) 
the relations of type (R4) (additional affine relations) are essential. 

Recall that the group $\t W_{\G}$  is obtained from $W_{\G}$ by omitting additional affine relations of type (R4). 

Our aim is to prove that  $\t W_{\G}$ is not invariant under mutations. 
%if $\G'$ is of type $W$ then the mutation class of $\G'$ contains
%is not isomorphic to $W$ for some diagram $\G$ in the corresponding mutation class. 
More precisely, we prove the following lemma.

\begin{lemma}
\label{different}
Let $\G$ be one of the diagrams shown in Table~\ref{add-t}, and let $W$ be the corresponding group from the right column of the table. Then $\t W_{\G}$ is not isomorphic to $W$. 

\end{lemma} 

Here is the plan of the proof. By Lemma~\ref{seven}, there is a surjective homomorphism $\varphi:\t W_{\G}\to W$. Our aim is to prove that $\varphi$ is not an isomorphism. According to Malcev~\cite[Theorem XII]{M}, this will imply that $\t W_{\G}$ is not isomorphic to $W$ as soon as $W$ is a finitely generated linear group, which is of course true for Coxeter groups. 

To show that $\varphi$ is not an isomorphism, we consider quotient groups $\t W_{\G}/H$ and $W/\varphi(H)$, where $H$ is the normal closure of a suitable element of $\t W_{\G}$, and see that these groups are not isomorphic.  

We deal with all the diagrams separately.

\subsection{${\mathbf{W=\t A_3}}$,}
$$\G=
\psfrag{4_}{\scriptsize $4$}
\psfrag{2_}{\scriptsize $$}
\psfrag{1}{\scriptsize $1$}
\psfrag{2}{\scriptsize $2$}
\psfrag{3}{\scriptsize $3$}
\psfrag{4}{\scriptsize $4$}
\raisebox{-0.8cm}{\epsfig{file=./pic/rel_B3.eps,width=0.13\linewidth}}
$$
Here 
$$\t W_{\G}=\langle t_1,t_2,t_3,t_4\,|\,t_i^2=(t_1t_2)^3=(t_2t_3)^3=(t_3t_4)^3=(t_4t_1)^3=(t_2t_4)^2=(t_3t_2t_1t_2)^3=(t_3t_4t_1t_4)^3=e\rangle,$$
$$W=\langle s_1,s_2,s_3,s_4\,|\,s_i^2=(s_1s_2)^3=(s_2s_3)^3=(s_3s_4)^3=(s_4s_1)^3=(s_2s_4)^2=(s_1s_3)^2=e\rangle,$$
the epimorphism $\varphi:\t W_{\G}\to W$ is defined by
$$
\varphi(t_1)=s_1,\ 
\varphi(t_2)=s_2,\
\varphi(t_3)=s_4s_2s_3s_2s_4,\
\varphi(t_4)=s_4
$$
Now consider the normal closure $H=\langle t_2t_4\rangle^{\t W_{\G}}$. Then the quotient group 
$$\t W_{\G}/H\cong\langle t_1,t_2,t_3\,|\,t_i^2=(t_1t_2)^3=(t_2t_3)^3=(t_3t_2t_1t_2)^3=e\rangle\cong\t A_2,$$
and
$$W/\varphi(H)=\langle s_1,s_2,s_3\,|\,s_i^2=(s_1s_2)^3=(s_2s_3)^3=(s_1s_3)^2=e\rangle\cong A_3$$
which are clearly not isomorphic.

\subsection{${\mathbf{W=\t D_n}}$,}
$$\G=
\psfrag{u}{\scriptsize $1$}
\psfrag{v}{\scriptsize $2$}
\psfrag{1}{\scriptsize $3$}
\psfrag{2}{\scriptsize $4$}
\psfrag{3}{\scriptsize $5$}
\psfrag{4}{\scriptsize $6$}
\psfrag{n}{\scriptsize $n+1$}
\psfrag{n1}{\scriptsize $n$}
\raisebox{-1.1cm}{\epsfig{file=./pic/rel_D.eps,width=0.2\linewidth}}
$$
\begin{multline*}\t W_{\G}=\langle t_1,\dotsm,t_{n+1}\,|\,t_i^2=(t_it_{i+1})^3=(t_1t_n)^3=(t_2t_{n+1})^3=(t_2t_n)^3=(t_it_j)^2\,({\text{otherwise}})=\\ =(t_1t_2t_nt_2)^2=(t_{n+1}t_2t_nt_2)^2=e\rangle,\end{multline*}
$$W\!=\!\langle s_1,\dots,s_{n+1}\,|\,s_i^2=(s_1s_2)^3\!=\dots=(s_{n-2}s_{n-1})^3\!=(s_2s_{n+1})^3\!=(s_{n-2}s_n)^3\!=(s_is_j)^2\,({\text{otherwise}})=e\rangle,$$
the epimorphism $\varphi:\t W_{\G}\to W$ is defined by
$$
\varphi(t_i)=s_i,\ {\text{for }}i\ne n, \qquad 
\varphi(t_n)=s_{n+1}s_1s_2\cdots s_{n-2}s_ns_{n-2}s_{n-3}\cdots s_1s_{n+1}
$$
Take $H=\langle t_1t_{n+1}\rangle^{\t W_{\G}}$. Then the quotient group 
$$\t W_{\G}=\langle t_1,\dotsm,t_{n}\,|\,t_i^2=(t_it_{i+1})^3=(t_1t_n)^3=(t_2t_n)^3=(t_it_j)^2\,({\text{otherwise}})=(t_1t_2t_nt_2)^2=e\rangle\cong\t A_{n-1},$$
while
$$W/\varphi(H)=\langle s_1,\dots,s_{n}\,|\,s_i^2=(s_1s_2)^3=\dots=(s_{n-2}s_{n-1})^3=(s_{n-2}s_n)^3=(s_is_j)^2\,({\text{otherwise}})=e\rangle\cong D_n$$

\subsection{${\mathbf{W=\t B_3}}$,}
$$\G=
\psfrag{4_}{\scriptsize $4$}
\psfrag{2_}{\scriptsize $2$}
\psfrag{1}{\scriptsize $1$}
\psfrag{2}{\scriptsize $2$}
\psfrag{3}{\scriptsize $3$}
\psfrag{4}{\scriptsize $4$}
\raisebox{-0.8cm}{\epsfig{file=./pic/rel_B3.eps,width=0.13\linewidth}}
$$
Here 
$$\t W_{\G}=\langle t_1,t_2,t_3,t_4\,|\,t_i^2=(t_1t_2)^4=(t_2t_3)^4=(t_3t_4)^3=(t_4t_1)^3=(t_2t_4)^2=(t_1t_2t_3t_2)^2=(t_3t_4t_1t_4)^3=e\rangle,$$
$$W=\langle s_1,s_2,s_3,s_4\,|\,s_i^2=(s_1s_4)^3=(s_2s_4)^4=(s_3s_4)^3=(s_1s_2)^3=(s_2s_3)^2=(s_1s_3)^2=e\rangle,$$
the epimorphism $\varphi:\t W_{\G}\to W$ is defined by
$$
\varphi(t_1)=s_2s_4s_1s_4s_2,\ 
\varphi(t_2)=s_2,\
\varphi(t_3)=s_4s_3s_4,\
\varphi(t_4)=s_3
$$
Now take $H=\langle (t_2t_3)^2\rangle^{\t W_{\G}}$. Then the quotient group 
$$\t W_{\G}=\langle t_1,t_2,t_3,t_4\,|\,t_i^2=(t_1t_2)^4=(t_2t_3)^2=(t_3t_4)^3=(t_4t_1)^3=(t_2t_4)^2=(t_1t_3)^2=e\rangle\cong B_4,$$
while
$$W/\varphi(H)=\langle s_1,s_2,s_3,s_4\,|\,s_i^2=(s_1s_4)^3=(s_2s_4)^2=(s_3s_4)^3=(s_1s_2)^3=(s_2s_3)^2=(s_1s_3)^2=e\rangle\cong A_1\times A_3$$

\subsection{${\mathbf{W=\t B_n}}$,}
$$\G=
\psfrag{u}{\scriptsize $n+1$}
\psfrag{2_}{\scriptsize $2$}
\psfrag{1}{\scriptsize $1$}
\psfrag{2}{\scriptsize $2$}
\psfrag{3}{\scriptsize $3$}
\psfrag{4}{\scriptsize $4$}
\psfrag{n}{\scriptsize $n-1$}
\psfrag{n+1}{\scriptsize $n$}
\raisebox{-0.8cm}{\epsfig{file=./pic/rel_B_.eps,width=0.23\linewidth}}
$$
\begin{multline*}\t W_{\G}=\langle t_1,\dotsm,t_{n+1}\,|\,t_i^2=(t_{n+1}t_1)^4=(t_{n+1}t_n)^4=(t_1t_2)^3=(t_2t_3)^3=\dots=(t_{n-1}t_{n})^3=\\ =(t_nt_1)^3=(t_it_j)^2\,({\text{otherwise}})=(t_{n+1}t_1t_nt_1)^2=e\rangle,\end{multline*}
\begin{multline*}W=\langle s_1,\dots,s_{n+1}\,|\,s_i^2=(s_{n+1}s_1)^4=(s_{1}s_{2})^3=(s_2s_{3})^3=\dots=(s_{n-2}s_{n-1})^3=\\ =(s_{n-2}s_n)^3=(s_is_j)^2\,({\text{otherwise}})=e\rangle,\end{multline*}
similarly to the case of $\t D_n$, the epimorphism $\varphi:\t W_{\G}\to W$ is defined by
$$
\varphi(t_i)=s_i,\ {\text{for }}i\ne n, \qquad 
\varphi(t_n)=s_{n+1}s_1s_2\cdots s_{n-2}s_ns_{n-2}s_{n-3}\cdots s_1s_{n+1}
$$
Take $H=\langle (t_1t_{n+1})^2\rangle^{\t W_{\G}}$. Then the quotient group 
$$\t W_{\G}=\langle t_1,\dotsm,t_{n}\,|\,t_i^2=(t_{n+1}t_n)^4=(t_1t_{2})^3=(t_2t_3)^3=\dots=(t_{n-1}t_n)^3=(t_it_j)^2\,({\text{otherwise}})=e\rangle\cong\t B_{n+1},$$
while
\begin{multline*}W/\varphi(H)=\langle s_1,\dots,s_{n+1}\,|\,s_i^2=(s_{n+1}s_1)^2=(s_{1}s_{2})^3=(s_2s_{3})^3=\dots=(s_{n-2}s_{n-1})^3=\\ =(s_{n-2}s_n)^3=(s_is_j)^2\,({\text{otherwise}})=e\rangle\cong B_n\times A_1\end{multline*}

\subsection{${\mathbf{W=\t G_2}}$,}
$$\G=
\psfrag{1}{\scriptsize $1$}
\psfrag{2}{\scriptsize $2$}
\psfrag{3_}{\scriptsize $3$}
\psfrag{3}{\scriptsize $3$}
\psfrag{4}{\scriptsize $4$}
\raisebox{-0.6cm}{\epsfig{file=./pic/rel-tg2.eps,width=0.11\linewidth}}
$$
Here 
$$\t W_{\G}=\langle t_1,t_2,t_3\,|\,t_i^2=(t_1t_2)^6=(t_1t_3)^6=(t_1t_2t_3t_2)^6=(t_2t_3t_1t_3)^6=(t_3t_1t_2t_1)^3=e\rangle,$$
$$W=\langle s_1,s_2,s_3,s_4\,|\,s_i^2=(s_1s_3)^6=(s_2s_3)^3=(s_1s_2)^2=e\rangle,$$
the epimorphism $\varphi:\t W_{\G}\to W$ is defined by
$$
\varphi(t_1)=s_3s_1s_3,\ 
\varphi(t_2)=s_3s_1s_3s_2s_3s_1s_3,\
\varphi(t_3)=s_3
$$
Now take $H=\langle (t_1t_3)^2\rangle^{\t W_{\G}}$. Then the quotient group 
$$\t W_{\G}=\langle t_1,t_2,t_3\,|\,t_i^2=(t_1t_2)^6=(t_1t_3)^2=(t_2t_3)^3=e\rangle\cong\t G_2,$$
while
$$W/\varphi(H)=\langle s_1,s_2,s_3\,|\,s_i^2=(s_1s_3)^2=(s_2s_3)^3=(s_1s_2)^2=e\rangle\cong A_1\times A_2$$

\section{Generalization for diagrams arising from unpunctured surfaces and orbifolds}
\label{unpunctured}

Let $\G$ be a diagram arising from an unpunctured surface or orbifold. We construct a group $W_{\G}$ in the similar way as before (but with one more additional type of relations, see Section~\ref{surf-construction}) and show that this group is invariant under mutations. 
In this case $W_{\G}$ is not a Coxeter group anymore, but a quotient of some Coxeter group (by relations of types (R3)--(R5), see below).

\subsection{Construction of the group $W_{\G}$}
\label{surf-construction}

\begin{defin}
\label{def surf}
Given a diagram $\G$ of order $n$ arising from a triangulated unpunctured surface or orbifold,
$W_\G$ is a group with  
\begin{itemize}
\item generators $s_1,\dots,s_n$ corresponding to the vertices of $\G$;
\item relations:
\begin{itemize}
\item[(R1)] $s_i^2=e$ for $i=1,\dots,n$;
\item[(R2)] $(s_is_j)^{m_{ij}}=e$ for all vertices $i$, $j$ not joined by an arrow labeled by $4$ (where $m_{ij}$ are defined in Section~\ref{aff group def});
\item[(R3)] cycle relation for every chordless oriented cycle (see relations of type (R3) in Section~\ref{aff group def});
\item[(R4)] four types of additional relations for affine diagrams from Table~\ref{add-t};
\item[(R5)] additional relations for a handle: 
$$
(s_{1}s_{2}s_{3}s_{4}s_{3}s_{2})^3=e
\text{\ and \ }
(s_{1}s_{2}s_{3}s_{4}s_{5}s_{4}s_{3}s_{2})^2=e
$$
for all subdiagrams  of type $\H_0$ and $\H$ shown in Fig.~\ref{rel-handle};
%\item[(6)] additional relations for small cases:
%if $\G$ coincides with the diagram $H^4$ or $O^5$ on Fig.~\ref{rel-handle} 
%one adds $(s_1s_2s_3s_4s_3s_2)^3=e$ or  ????? respectively.
%<kazhetsya, i bez nih horosho>

\end{itemize}

\end{itemize}
\end{defin}

\begin{figure}[!h]
\begin{center}
\psfrag{H5}{ $\H$}
\psfrag{O5}{ $O^5$}
\psfrag{H4}{ $\H_0$}
\psfrag{4}{\scriptsize $4$}
\psfrag{2}{\scriptsize $2$}
\psfrag{s1}{\scriptsize ${1}$}
\psfrag{s2}{\scriptsize ${2}$}
\psfrag{s3}{\scriptsize ${3}$}
\psfrag{s4}{\scriptsize ${4}$}
\psfrag{s5}{\scriptsize ${5}$}
\epsfig{file=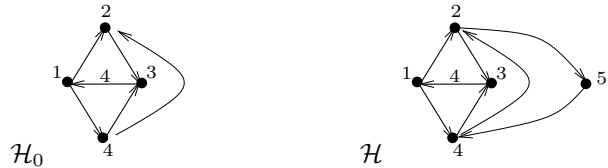,width=0.5\linewidth}
\caption{Additional relations: the diagram $\H$ corresponds to a handle with two marked points at the boundary component. The diagram $\H_0$ corresponds to a handle with one marked point (i.e., the boundary corresponds to vertex $5$ of $\H$).}
% also known as {\it dreadful torus}.}
\label{rel-handle}
\end{center}
\end{figure}

\begin{remark}
\label{HH}
(i) Relations (R1)-(R4) are the same as in the construction of the group for the affine case.

\medskip
\noindent
(ii) The diagrams for relations (R5) correspond to a handle with one and two marked points at the boundary component.

\medskip
\noindent
(iii) The diagram $\H_0$ is a subdiagram of $\H$, and the relation (R5) for $\H_0$ is a corollary of the relation (R5) for $\H$. Furthermore, it is easy to observe that if the diagram $\H_0$ is a subdiagram of a bigger diagram $Q$ originating from a triangulation of a surface or orbifold, then there exists $\H\subset \Q$ containing $\H_0$. Together with the observation above this implies that the only diagram for which the first relation in (R5) needs to be applied is $\H_0$ itself.

%\medskip
%\noindent
%(iii) The diagrams in condition (6) correspond to ``small surfaces'': 
%the torus with one hole and one marked point on the boundary and to the sphere with one puncture and orbifold point.
%
%The diagram $O^5$ can not be a subdiagram of any bigger diagram arising from a triangulated orbifold
%(its only block decomposition contains no open vertices and corresponds to a one punctured sphere with one orbifold point whic%h is not a part of any other connected orbifold).
%
%If the diagram $H^4$ is a subdiagram of a diagram $D$ arising from an unpunctured triangulated surface 
%or orbifold then there exists a vertex $z$ such that a subdiagram $H^4\cup z$ is of the type $H^5$.

\medskip
\noindent
(iii)
The second relation (R5) is equivalent to any of the following three relations:
$$
(s_{1}s_{4}s_{3}s_{2}s_{5}s_{2}s_{3}s_{4})^2=e,\qquad
(s_{3}s_{2}s_{1}s_{4}s_{5}s_{4}s_{1}s_{2})^2=e,\qquad
(s_{3}s_{4}s_{1}s_{2}s_{5}s_{2}s_{1}s_{4})^2=e.
$$

\end{remark}

%The answer to the following question would be interesting:

%\begin{question}
%Are the two relations in (R5) independent?

%\end{question}

\subsection{Invariance of the group $ W_{\G}$}

\begin{theorem}
\label{thm surf/orb}
Let $\G$ be a diagram arising from an unpunctured surface or orbifold, and let 
 $W_{\G}$ be the group defined as above. Then $W_{\G}$ is invariant under mutations of $\G$.

\end{theorem}
  
Let us define pseudo-cycles and risk diagrams in the same way as for affine diagrams: {\it pseudo-cycles} are supports of relations (R1)--(R5), and {\it risk diagrams} are diagrams of the form $x\cup\P$, where $\P$ is a pseudo-cycle, and $x$ is connected to $\P$ by at least one incoming and one outgoing arrow.

 Now note that the proofs of Lemmas~\ref{in and out} and~\ref{outside} do not use the property of $\G$ to be of affine type. Therefore, to prove Theorem~\ref{thm surf/orb} we can use exactly the same strategy as in the affine case: we list all pseudo-cycles, find all risk subdiagrams for each of them and check  conditions {\rm{(C1)}} and {\rm{(C2)}} of Lemma~\ref{in and out}. 

\begin{lemma}
\label{cycle->puncture}
Let $\G$ be a diagram arising from a triangulated unpunctured surface.
Then $\G$ contains no oriented chordless cycles of length bigger than 3.

Moreover, the same holds for diagrams arising from triangulated unpunctured orbifolds.

\end{lemma}

\begin{proof}
First suppose that $\G$ comes from a triangulated surface.
Then $\G$ is block-decomposable. Since the surface is unpunctured, the list of possible blocks in the decomposition is exhausted by blocks of type I and II (both corresponding to ordinary, non-self-folded triangles).
If these blocks are arranged to make an oriented cycle (not composing a single block) then the corresponding triangles make a circular neighborhood of a common vertex (see Fig.~\ref{puncture}), so this turns into a puncture which is not allowed by the assumption.

Now, suppose that $\G$ comes from a triangulated unpunctured orbifold. Then the block-decomposition of $\G$ consists of blocks
of types I, II, $\widetilde{\rm{IV}}$ and $\widetilde{\rm{V}}_{12}$. Furthermore, if $\C\subset \G$ is an oriented cycle, then  no block of type
 $\widetilde V_{12}$ has an arrow in $\C$ (since this block has only one white vertex).
Let $\C'$ be a subdiagram of $\G$ spanned by all blocks having an arrow in $\C$. Constructing a triangulation corresponding to $\C'$ we get a puncture again, see  Fig.~\ref{puncture}.
% (to convince youself that the blocks of type  $\widetilde{\rm{V}}_{12}$
%make no difference for the statement one could cut the orbifold along the pending edge turning the orbifold triangle into an 
%ordinary one with two identified vertices, the cutting turns the block of type  $\widetilde{\rm{V}}_{12}$ into a block of type I). 

\end{proof}

\begin{figure}[!h]   
\begin{center}
\epsfig{file=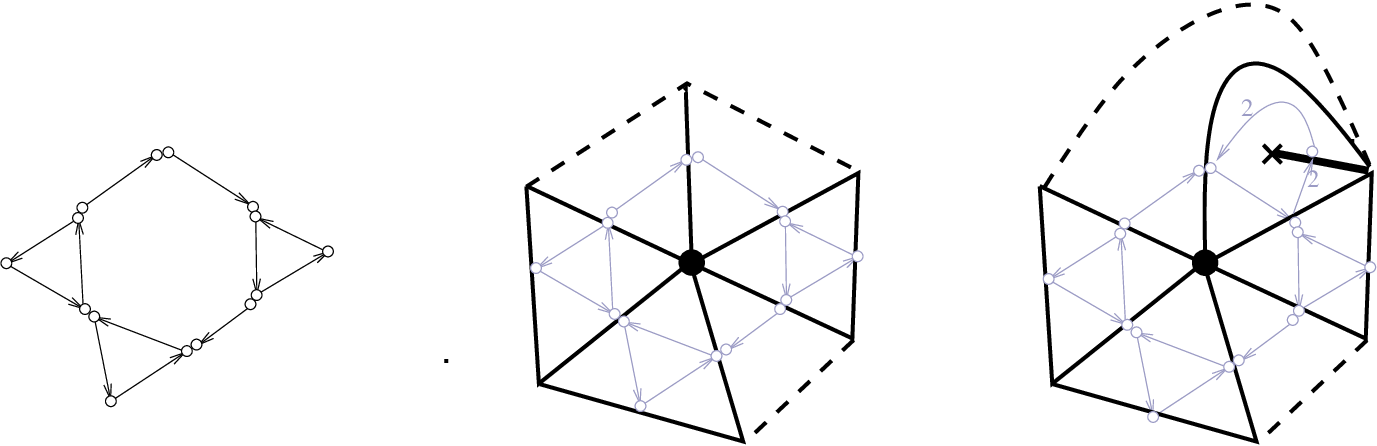,width=0.9\linewidth}
\caption{Lemma~\ref{cycle->puncture}: A long oriented cycle in the diagram corresponds to a puncture on the surface/orbifold (some of the vertices or edges of the triangles in the figure may coincide).}
\label{puncture}
\end{center}
\end{figure}

In view of Lemma~\ref{cycle->puncture}, any pseudo-cycle in $\G$  is either a subdiagram of order at most $3$, or of one of four additional 
(affine) types in Table~\ref{add-t}, or the diagrams $\H_0$ and $\H$ in Fig.~\ref{rel-handle}. 
%of the three diagrams in Fig... ($H^5$, $H^4$ and $O^5$).
Moreover, three of five  additional affine types are diagrams of mutation type $\widetilde D_n$ or $\widetilde B_n$, thus ones arising from a punctured surface/orbifold. 
One is of mutation type $\widetilde G_2$, so does not arise from surfaces or orbifolds.
Hence, in the unpunctured case we only need to check the following types of pseudo-cycles:
\begin{itemize}
\item two-vertex subdiagrams;  
\item oriented triangles; 
\item additional affine pseudo-cycle of mutation type $\widetilde A_{2,2}$;
\item diagrams $\H_0$ and $\H$. %, $H^4$ and $O^5$.

\end{itemize}

\begin{lemma}
\label{1 for handle}
Condition {\rm{(C1)}} of Lemma~\ref{in and out} holds for $\H_0$ and $\H$. %, $H^4$ and $O^5$.

\end{lemma}

The proof of the lemma is straightforward.
Together with the result of Lemma~\ref{condition1} the lemma implies that {\rm{(C1)}} holds for all 
pseudo-cycles that can be found in a diagram arising from an unpunctured surface/orbifold.

Our next step is to find all risk diagrams for all pseudo-cycles.

\begin{lemma}
\label{risk for handle}
Let $\G$ be a diagram  arising from an unpunctured surface/orbifold, and let $\R$ be its risk subdiagram.
Then $\R$ is either a risk diagram for some affine diagram, or $\R=\H$  or $\R=\mu_5(\H)$ 
(where  $\mu_5(\H)$ is the diagram on Fig.~\ref{mu5H5} obtained from $\H$ by one mutation).
%Then either is a risk subdiagram for some affine diagram.
%or $\R$ coincides with one of $O^5$, $H^5$ and $\mu_5(H^5)$ (where  $\mu_5(H^5)$ is the diagram on Fig.~\ref{mu5H5} obtained from $H^5$ by one mutation).

\end{lemma} 

\begin{proof}
Let $\P$ be a pseudo-cycle and $\R=\!x\cup\P\!$. 
If $\P$ has two or three vertices we list all block-decomposable diagrams with 3 or 4 vertices respectively
(and choose those of them having $\P$ as a subdiagram) and verify explicitly that they all appear as subdiagrams of diagrams of affine type.

For $\P$ of type $\widetilde A_{2,2}$ or of type  $\H$ %, $H^4$ and $O^5$ 
we note that $\P$ has an arrow labeled by $4$, so this arrow is obtained by gluing two blocks. 
Keeping in mind that $\R$ is block-decomposable and that the vertex $x$ of a risk diagram should be connected to $\P$ by both an incoming and an outgoing arrow, it is easy to see that the pseudo-cycle $\H$ does not belong to any risk diagram, and the pseudo-cycle $\widetilde A_{2,2}$ %and $H^4$ 
belongs to the risk diagram $\mu_5(\H)$ only.

Finally, for $\P$ of type $\H_0$, $\P$ is contained in $\H$ (see Remark~\ref{HH}(iii)), and the only vertex of $\G$ connected to vertices of $\P$ is the remaining vertex of $\H$: this can be easily seen from the block decomposition. This implies that the risk diagram coincides with $\H$.

\end{proof}

\begin{figure}[!h]
\begin{center}
\psfrag{mu5H5}{ $\mu_5(\H)$}
\psfrag{4}{\scriptsize $4$}
\psfrag{2}{\scriptsize $2$}
\psfrag{s1}{\scriptsize $1$}
\psfrag{s2}{\scriptsize $2$}
\psfrag{s3}{\scriptsize $3$}
\psfrag{s4}{\scriptsize $4$}
\psfrag{s5}{\scriptsize $5$}
\epsfig{file=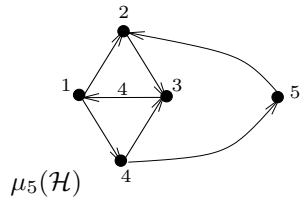,width=0.25\linewidth}
\caption{The diagram $\mu_5(\H)$ obtained by a mutation of $\H$.}
\label{mu5H5}
\end{center}
\end{figure}

\begin{lemma}
\label{2 for handle}
Condition {\rm{(C2)}} of Lemma~\ref{in and out} holds for all risk subdiagrams of diagrams arising from unpunctured surfaces and orbifolds.

\end{lemma}

\begin{proof}
By Lemma~\ref{risk for handle}, we only need to check risk diagrams of type  $\H$ and $\mu_5(\H)$.
%Since  $O^5$ and $H^5$ are also pseudo-cycles and $\mu_5(H^5)$ is obtained from $H^5$ by one mutation 
Thus, {\rm{(C2)}} for this risk diagram is already checked as {\rm{(C1)}} for pseudo-cycle  $\H$.

\end{proof}

Lemmas~\ref{1 for handle} and~\ref{2 for handle} imply Theorem~\ref{thm surf/orb}.

%\begin{remark}
%????????????????
%Nuzhno li chto-to pro $H^4$ i $O^5$?
%Sootnosheniya ne neobhodimi, t.k. eti mutacionnie classi sostoyat tol'ko iz odnogo kolchana.
%Poverhnosti: ruchka s odnoy tochkoy na granice i sphera s odnim prokolom i odnoy orbifoldnoy tochkoy.
%??????????
%
%V $H^4\in H^5$ ne trebuetsya sootnosheniya sootvetstvuyushego $H^4$.
%\end{remark}

\begin{remark}
Unlike to the affine case, the group  $W_{\G}$ for $\G$ arising from an unpunctured surface or orbifold  is not a Coxeter group but a 
quotient of some Coxeter group. 

\end{remark}

\begin{question}
(i)
Given two mutationally non-equivalent diagrams $\G_1$ and $\G_2$ arising from (distinct) unpunctured surfaces or orbifolds,
is it true that  $W_{\G_1}$ is not isomorphic to  $W_{\G_2}$?

(ii) What types of groups can be obtained as groups $W_{\G}$? 
%We note that all the groups $W_{\G}$ are infinite (except for $\G$ of finite type).

\end{question}

%\begin{question}
%Is it possible to generalize Theorem~\ref{thm surf/orb} to the case of punctured surfaces/orbifolds?
%(possibly with introducing some more/another sorts of relations).

%\end{question}

\section{Exceptional diagrams} 
\label{except}

In this section, we construct the groups for the remaining exceptional mutation-finite diagrams, i.e. for diagrams which are neither 
block-decomposable nor of finite or affine type. By Theorem~\ref{class}, these diagrams are exhausted by the following mutation types:
$X_6$, $X_7$, $E_6^{(1,1)}$, $E_7^{(1,1)}$, $E_8^{(1,1)}$, $G_2^{(*,+)}$, $G_2^{(*,*)}$, $F_4^{(*,+)}$, $F_4^{(*,*)}$.

\begin{defin}[Group $W_\G$ for exceptional diagrams]
\label{def-exc} 
Let $\G$ be a diagram of an exceptional mutation type.
Define group $W_\G$ as the group with generators $s_1,\dots,s_n$ corresponding to the vertices of $\G$ and with relations  
\begin{itemize}
\item[(R1)] $s_i^2=e$ for $i=1,\dots,n$;
\item[(R2)] $(s_is_j)^{m_{ij}}=e$ for all vertices $i$, $j$ not joined by an arrow labeled by $4$ (where $m_{ij}$ are defined in Section~\ref{aff group def});
\item[(R3)] cycle relation for every chordless oriented cycle (see relations of type (R3) in Section~\ref{aff group def});
\item[(R4)] (additional affine relations) for every subdiagram of $\G$ of the form shown in the first column of Table~\ref{add-t} we take the relations listed in the second column of the table; 
\item[(R5*)] additional $X_5${\it -relation}
$$(s_1s_0s_2s_0s_1s_3s_0s_4s_0s_3)^2=e
$$
for diagram $X_5$ shown in Fig.~\ref{X5}.

\end{itemize}

\end{defin}

\begin{figure}[!h]
\begin{center}
\psfrag{X5}{ $X_5$}
\psfrag{4}{\scriptsize $4$}
\psfrag{2}{\scriptsize $2$}
\psfrag{s1}{\scriptsize $1$}
\psfrag{s2}{\scriptsize $2$}
\psfrag{s3}{\scriptsize $3$}
\psfrag{s4}{\scriptsize $4$}
\psfrag{s0}{\scriptsize $0$}
\epsfig{file=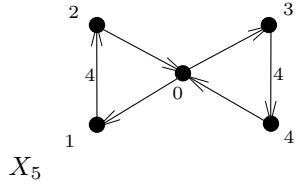,width=0.25\linewidth}
\caption{A diagram $X_5$.}
\label{X5}
\end{center}
\end{figure}

\begin{remark}
(i) For non-decomposable diagrams of finite or affine type the definition above coincides with ones from~\cite{BM} and Section~\ref{aff group def}. 

\medskip
\noindent
(ii) The relation (R5*) is equivalent to $(s_2s_0s_1s_0s_2s_4s_0s_3s_0s_4)^2=e$.

\medskip
\noindent
(iii) Relation (R5*) is necessary for mutation classes $X_6$ and $X_7$ only.

\medskip
\noindent
(iv) The diagram $X_5$ corresponds to a triangulated punctured annulus. We expect that relation (R5*) will lose its exceptional character when we will define the group $W_\G$ for surfaces with punctures.
 
\end{remark}

\begin{theorem}
\label{thm exc}
If $\G$ is a diagram of the exceptional finite mutation type 
(i.e. $\G$ is mutation-equivalent to one of $X_6,X_7,E_6^{(1,1)},E_7^{(1,1)},E_8^{(1,1)},G_2^{(*,+)},G_2^{(*,*)},F_4^{(*,+)}$ or
$F_4^{(*,*)}$)
then  the group  $W_{\G}$ is invariant under mutations.
\end{theorem}

Note that, similarly to the groups constructed for surfaces or orbifolds, the groups obtained in the exceptional cases are quotients of Coxeter 
groups. We do not know whether these groups are distinct for different mutation classes or not.

To prove Theorem~\ref{thm exc} we consider cases of $X_n$, $E_n^{(1,1)}$, $G_2^{(\cdot,\cdot)}$ and $F_4^{(\cdot,\cdot)}$ separately.

\subsection{Groups for $X_6$ and $X_7$}

The proof of the invariance of the group $W_\G$ under mutations 
 is a straightforward check of pseudo-cycles and risk diagrams based on Lemma~\ref{in and out}.

More precisely, first we check condition {\rm{(C1)}} for the pseudo-cycle of type $X_5$.
Then we check that there is no risk diagrams containing the pseudo-cycle $X_5$: we look through the mutation classes of $X_6$ and $X_7$ using the fact the they are small (containing 5 and 2 diagrams respectively).
We also check that if $\R$ is a risk diagram containing some pseudo-cycle then either $\R$ is a risk diagram for some diagram of affine type 
or $\R=X_5$. 
Condition {\rm{(C2)}} for risk subdiagrams of diagrams of affine type has already been checked above.
{\rm{(C2)}} for the risk diagram $\R=X_5$ is {\rm{(C1)}} for the pseudo-cycle of type $X_5$. 

%\begin{remark}
%In the proof of Theorem~\ref{x7} we don't need to make big case by case study since it is easy to see that no psudocycle
%contains more than 4 vertices, which implies that no risk subdiagram contains more than 5 vertices. All mutaion finite 
%(sqew-symmetric) diagrams
%of order at most 5 are block-decomposable
%
%
%\end{remark}

\subsection{Groups for diagrams  $G_2^{(*,+)}$ and $G_2^{(*,*)}$}

The proof of the invariance is a direct check due to small mutation classes (6 and 2 diagrams respectively).

%\begin{remark}
%The relations of type (4) (additional affine relations) are irrelevant here since no diagram in the mutation class contain any pseudocycle of type (4), this is immediately follows from the fact that $|G_2^{(*,+)}|=|G_2^{(*,*)}|=4$ which is equal to 
%the minimal size of a diagram of type 4.

%\end{remark}

\subsection{Groups for diagrams  $F_4^{(*,+)}$ and $F_4^{(*,*)}$}

The mutation classes of $F_4^{(*,+)}$ and $F_4^{(*,*)}$ are rather large (90 and 35 diagrams respectively), so we use pseudo-cycles and risk diagrams. 

More precisely, if $\P$ is a pseudo-cycle and it is not a subdiagram of any affine diagram, then $\P$ defines a relation of type (R3) (cyclic relation) and is one of the cycles listed in Table~\ref{short cycles}.
There is a unique pseudo-cycle which is not a subdiagram of any affine diagram and does not contain arrows labeled by $3$, namely the cyclic diagram shown in row 10 of the table. A straightforward computation shows that {\rm{(C1)}} holds for this pseudo-cycle.

Now, we need to list and check all risk diagrams. 

\begin{lemma}
\label{l f__}
Condition {\rm{(C2)}} holds for all risk subdiagrams of  $F_4^{(*,+)}$ and $F_4^{(*,*)}$.
\end{lemma}

\begin{proof}
First, we do not need to check any decomposable risk diagrams 
(by results of Sections~\ref{sec a},~\ref{sec d} and~\ref{sec bc}) or subdiagrams of affine diagrams. This implies that we are not interested in risk diagrams of size smaller than 6
(since any diagram of size at most 5 and  containing no arrows labeled by $3$ is either block-decomposable or a subdiagram of 
$\widetilde F_4$). So, we need to study risk diagrams of order 6 only, i.e. the diagrams mutation-equivalent to 
 $F_4^{(*,+)}$ or $F_4^{(*,*)}$.

To check risk subdiagrams of order 6 we consider all pseudo-cycles of order 5 and add an additional vertex $x$ to them.
There are 4 pseudo-cycles of order 5, namely a simply-laced cycle, the cyclic diagram of mutation type $\widetilde F_4$ (shown in row 9 of Table~\ref{short cycles}), and additional affine pseudo-cycles of types $\widetilde D_4$ and $\widetilde B_4$. 
For each pseudo-cycle $\P$ we add a vertex $x$ such that 
\begin{itemize}
\item $x\cup \P$ is mutation-finite;
\item $x$ is connected to $\P$ by at least one outgoing and at least one incoming arrow;
\item $x\cup \P$ has at least one arrow labeled by $2$ (otherwise we get either block-decomposable diagram, or $E_6$ or $X_6$).

\end{itemize}

The mutation-finiteness of $\R=x\cup \P$ implies in particular that
\begin{itemize}
\item $x$ is connected to $\P$ by arrows labeled by $1$, $2$ or $4$ only;
\item all non-oriented cycles in $\R$ are simply-laced (see  Remark~\ref{non-or cycles} below).

\end{itemize}

It turns out after a short case-by-case study that any mutation-finite diagram $\R =x\cup \P$ of the required type is
either block-decomposable (which is not the case for the diagram of mutation type  $F_4^{(*,+)}$ or $F_4^{(*,*)}$)
or the diagram shown in Fig.~\ref{f__}. The mutation $\mu_x$ turns the latter diagram into the cyclic diagram
of mutation type $F_4^{(*,+)}$ (row 10 in Table~\ref{short cycles}), so  {\rm{(C2)}} for this risk diagram
was checked as a  {\rm{(C1)}} for the cyclic pseudo-cycle.

\end{proof}

\begin{figure}[!h]
\begin{center}
\psfrag{x}{\scriptsize $x$}
\psfrag{2}{\scriptsize $2$}
\epsfig{file=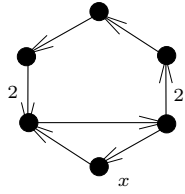,width=0.15\linewidth}
\caption{To the proof of Lemma~\ref{l f__}.}
\label{f__}
\end{center}
\end{figure}

\begin{remark}
\label{non-or cycles}
It is an easy observation that all non-oriented mutation-finite cycles are simply-laced. In skew-symmetric case this was mentioned in~\cite{S1}. 
%\item by noticing that such a diagram is either decomposable (in which case the statement is easy)
%or mutation equivalent to $F_4$, $\widetilde F_4$, $F_4^{(*,+)}$ or $F_4^{(*,*)}$ (in this case one need to look through the representatives of the mutation classes having small number of vertices, which is easy using the applet~\cite{Kel}).  

\end{remark}

\subsection{Groups for diagrams  $E_6^{(1,1)}$, $E_7^{(1,1)}$ and  $E_8^{(1,1)}$}

To check the invariance of the groups we consider pseudo-cycles and risk diagrams.

\begin{lemma}
Conditions {\rm{(C1)}} and {\rm{(C2)}} hold for all pseudo-cycles and all risk diagrams of
  $E_6^{(1,1)}$, $E_7^{(1,1)}$ and  $E_8^{(1,1)}$.

\end{lemma}

\begin{proof}
Condition {\rm{(C1)}} holds since we have not introduced any new pseudo-cycles (comparing to the affine case).

To prove that {\rm{(C2)}} holds note that the diagrams $E_6^{(1,1)}$, $E_7^{(1,1)}$ and  $E_8^{(1,1)}$ are skew-symmetric,
which implies that  any pseudo-cycle is of one of the following forms:
\begin{itemize}
\item a simply-laced cycle; 
\item a cycle of type $\widetilde A_{2,1}$ (row 1 in Table~\ref{short cycles});
\item an additional affine pseudo-cycle of type $\widetilde A_{2,2}$; 
\item an additional affine pseudo-cycle of type $\widetilde D_n$. 

\end{itemize}

The risk diagrams containing pseudo-cycles of types  $\widetilde A_{2,1}$  and  $\widetilde A_{2,2}$ can be checked explicitly.
The risk diagrams for a simply-laced cycle and an additional affine pseudo-cycle of type $\widetilde D_n$ 
are described in  Remark~\ref{bl-I} and also can be easily checked.

\end{proof}

This completes the proof of Theorem~\ref{thm exc}.

\end{document}